\documentclass[reqno]{amsart}
\usepackage{amssymb,amsthm,amsfonts,amstext,amsmath}
\usepackage{newcent}         
\usepackage{helvet}          
\usepackage{courier}         
\usepackage{stmaryrd}
\usepackage{graphicx}
\usepackage{enumerate}
\numberwithin{equation}{section}
\usepackage{color}
\usepackage{float}
\usepackage{mathrsfs}
\usepackage[T1]{fontenc}
\usepackage{url}

\usepackage{tikz}

\tikzset{
  big arrow 3/.style={
    decoration={markings,mark=at position 0.7 with {\arrow[scale=0.7,#1]{>}}},
    postaction={decorate}},
  big arrow/.default=black}

\usepackage{float}
\usepackage{dsfont}
\usepackage[colorlinks=true,linkcolor=blue,citecolor=magenta]{hyperref}

\usepackage{tikz}
\usetikzlibrary{arrows,decorations.pathmorphing}
\usetikzlibrary{calc, shapes, backgrounds,arrows.meta,patterns,fadings}
\usetikzlibrary{decorations.markings,positioning}

\def\centerarc[#1](#2)(#3:#4:#5){\draw[#1] ($(#2)+({#5*cos(#3)},{#5*sin(#3)})$) arc (#3:#4:#5);}
\tikzset{
  big arrow 3/.style={
    decoration={markings,mark=at position 0.7 with {\arrow[scale=0.7,#1]{>}}},
    postaction={decorate}},
  big arrow/.default=black}
  
\def\centerarc[#1](#2)(#3:#4:#5){\draw[#1] ($(#2)+({#5*cos(#3)},{#5*sin(#3)})$) arc (#3:#4:#5);}

\newtheorem{theorem}{Theorem}[section]
\newtheorem{lemma}[theorem]{Lemma}
\newtheorem{proposition}[theorem]{Proposition}

\newtheorem{remark}[theorem]{Remark}
\newtheorem{definition}[theorem]{Definition}

\newcommand{\one}{\mathds{1}}
\newcommand{\mc}[1]{{\mathcal #1}}

\newcommand{\bb}[1]{{\mathbb #1}}

\newcommand{\ignore}[1]{}

\keywords{Reinforced random walk, ant random walk, directed edges, random walk on graphs}

\begin{document}

\title{Stochastic processes with competing reinforcements}

\author[D. Erhard]{Dirk Erhard}
\address{UFBA\\
 Instituto de Matem\'atica, Campus de Ondina, Av. Adhemar de Barros, S/N. CEP 40170-110\\
Salvador, Brazil}
\curraddr{}
\email{erharddirk@gmail.com}
\thanks{}

\author[G. Reis]{Guilherme Reis}
\address{TUM\\
 Technical University of Munich, Faculty of Mathematics, Boltzmannstrasse 3, 85748 Garching bei München, Germany}
\curraddr{}
\email{guilherme.dwg@gmail.com, guilherme.reis@tum.de }
\thanks{}

\subjclass[2010]{60K35, 60K37, 60G50}

\begin{abstract}

We introduce a simple but powerful strategy to study processes driven by two or more reinforcement mechanisms in competition. We apply our method to two types of models: to non conservative zero range processes on finite graphs, and to multi-particle random walks with positive and negative reinforcement on the edges. The results hold for a broad class of reinforcement functions, including those with superlinear growth.  Our strategy consists in a comparison of the original processes with suitable reference models. To implement the comparison we estimate a Radon-Nikodym derivative on a carefully chosen set of trajectories.  Our results describe the almost sure long time behaviour of the processes. We also prove a phase transition depending on the strength of the reinforcement functions.
  
\end{abstract}

\maketitle


\allowdisplaybreaks

\section{Introduction}

\subsubsection*{Description of the models.}

In this work we consider two stochastic processes. Our first model is a \emph{Non Conservative Zero Range process} (NCZR) on a finite graph $G$.  This process starts with a finite number of walkers and is described by a field  $(\eta(n))_{n\geq 0}=(\eta_v(n):v\in G)_{n\geq 0}$ of natural numbers that represent the quantity of particles at each vertex $v\in G$ at a given discrete time $n$. Given three functions $W_1, W_2, W_3:\bb N\to\bb R_{+}$ the dynamics of $\eta$ at time $n+1$ evolves as follows: 
\begin{itemize}
	\item A particle is created at site $v$ with a probability proportional to $W_1(\eta_v(n)).$
	\item A particle is annihilated at site $v$ with a probability proportional \\ to $W_2(\eta_v(n))$.
	\item A particle jumps from $v$ to a neighbour $w$ with a probability proportional to $W_3(\eta_v(n))$.
\end{itemize}   

See Section~\ref{sec:NCZR} for a complete definition of the NCZR.

In the second model, called the \emph{Multi Particle Ant Random Walk} (Ant RW),
which itself is an adaptation of a similar model introduced by Erhard, Franco and Reis in~\cite{EFR2019},  we consider a finite quantity of walkers $(X_i(n)\,:\,1\leq i\leq N)_{n\geq 0}$ that interact through a dynamical environment on a finite graph $G=(V,E)$, given by a field $\{c_n(\vec{e}),\, n\geq 0\}_{\vec{e}}$ indexed by the set of oriented edges of $G$. Given  $\vec{e}=(e_-,e_+)\in E$, the quantity $c_n(\vec{e})$ is the number of jumps of the collection of walkers $(X_i)_i$ over $\vec{e}$ minus the number of jumps over the reversed edge $(e_+,e_-)$ up to time $n$. We refer to this quantity in the sequel as the crossing number of the edge $\vec{e}$. Observe that $c_n(\vec{e})$ can be either negative or positive. Assume that the walker $X_j$ is at time $n$ on the vertex $e_-$. Then it will jump to $e_+$ with a probability proportional to:
\begin{itemize}
	\item  $W_1(c_{n}(\vec{e}))$ if $c_{n}(\vec{e})\geq 0$.
	\item   $W_2(-c_{n}(\vec{e}))$ if $c_{n}(\vec{e})<0.$
\end{itemize}

Here, $W_1:\bb N\to\bb R_{+}$ and $W_2:\bb N\to\bb R_{+}$ are two given functions that modulate the behaviour of the model. We refer to Section~\ref{sec:AntRW} for a full description of that process.
\begin{center}
	\textit{Question:} For which choices of functions $W_1, W_2$ and $W_3$ do the two models above localise? 
\end{center}
To get a better feeling for what we mean by localisation we will provide a loose statement of our main results for a particular class of functions $W_1, W_2$ and $W_3$. Rigorous statements are deferred to later sections.

To that end fix $p,q,r \in \bb R$ and define the three functions alluded to above via $W_1(k)=k^p$, $W_2(k)=k^q$, and $W_3(k)=k^r.$ We introduce the following events  which we will use only in the context of the NCZR:
\begin{itemize}
	\item ${\rm Creat}$ is the event that   after some finite time $n$, there are no jumps or annihilation of particles, only creations.

	\item ${\rm Mon}$ is the event that there is a unique site $w$ such that eventually for all large $n$ the only transformation in the system takes the form $\eta_w(n+1)=\eta_w(n)+1$ (monopoly).
\end{itemize}

Note that if ${\rm Mon}$ occurs then the system is indeed completely localised in the sense that there exists a time  $m$ such that after $m$ for all sites $v\neq w\in G$ the number of particles at $v$ does not change. If ${\rm Creat}$ occurs after some time $m$ particles are only created, never annihilated and do not swap sites.

\begin{theorem}\label{thm:NCZR} One has the following phase transition depending on $p,q,r\in \bb R:$
	\begin{enumerate}
		\item if $p\leq 1+\max\{q,r\}$ then $\bb P({\rm Creat})=0,$
		\item if $p>1+\max\{q,r,0\}$ then $\bb P({\rm Creat})=1,$ and $\bb P({\rm Mon})=1$.
	\end{enumerate}
\end{theorem}

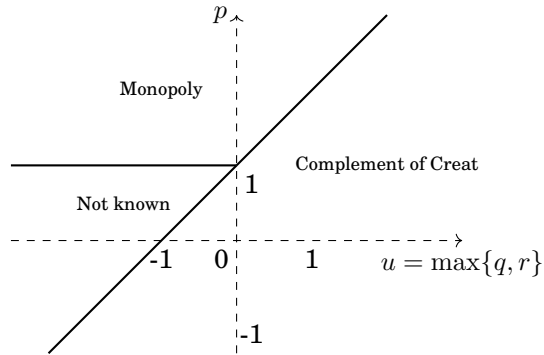
\begin{figure}[h]
\begin{tikzpicture}

\draw[->, dashed] (-3,0) -- (3,0) node[anchor=north] {$u=\max\{q,r\}$};
\draw	(-0.2,0) node[anchor=north] {0}
		(1,0) node[anchor=north] {1}
		(-1,0) node[anchor=north] {-1};
\draw	(-1,2) node{{\scriptsize Monopoly}}
		(2,1) node{{\scriptsize Complement of Creat}}
		(-1.5,0.5) node{{\scriptsize Not known}};

\draw[->,dashed] (0,-1.5) -- (0,3) node[anchor=east] {$p$};

\draw
		(0.2,1) node[anchor=north] {1}
		(0.2,-1) node[anchor=north] {-1};

\draw[thick] (-3,1) -- (0,1);
\draw[thick] (-2.5,-1.5) -- (0,1);
\draw[thick] (0,1) -- (2,3);

\end{tikzpicture}
      \label{fig:a}
\caption{Phase transition diagram.}
\end{figure}
Notice in Figure~\ref{fig:a} that for the case $\max\{q,r\}\leq 0$ we assume $p>1$ to prove that the Monopoly event happens with probability one.  Furthermore, the Monopoly event happens with probability one if $\max\{q,r\}>0 $ and $p>1+\max\{q,r\}$. In the case $p\leq 1+\max\{q,r\}$, Theorem~\ref{thm:NCZR} says that the Creation event does not happen almost surely.
Observe that Theorem~\ref{thm:NCZR} does not cover the case $\max\{q,r\}\leq 0$ and and $1+\max\{q,r\}< p\leq 1$.

Now we loosely state our main result for the Ant RW with $N\geq 1$ walkers in the special case that $W_1(k)=k^p$ and $W_2(k)=k^q.$   We start with some loose definitions to provide an idea, precise statements are postponed to Section~\ref{sec:AntRW} :
\begin{itemize}
\item We call a circuit $C$ a trapping circuit if there is a walker that eventually will spin around $C$ forever without leaving it.
\end{itemize}  

  Similarly to the previous model we introduce two events  that will be used only in the context of the Ant RW:
\begin{itemize}
	\item ${\rm Monot}$ is the event that for any directed edge $\vec{e}$ the map  $n\mapsto c_n(\vec{e})$ is monotone (incresing or decreasing) for large times. 
	\item ${\rm LocTrapp}$ is the event that ${\rm Monot}$ occurs and there are disjoint circuits $(C_i)_{1\leq i\leq k}$that are trapping circuits. Moreover, any circuit that is crossed infinitely many times is a trapping circuit.
\end{itemize}
Note that if ${\rm LocTrapp}$ occurs, then the walkers eventually localize (or are trapped) in the circuits $(C_i)_{1\leq i\leq k}$. 
\begin{theorem}\label{thm:example2}Assume the the graph $G$ has at least two circuits. One has the following phase transition depending on $p,q\in \bb R:$
	\begin{enumerate}
		\item if $p\leq 1+q$ then $\bb P({\rm Monot})=0,$
		\item if $p>1+\max\{q,0\}$ then  $\bb P({\rm LocTrapp})=1.$
	\end{enumerate}
\end{theorem}  

\begin{figure}[h]
\label{fig:theorem2}
\begin{tikzpicture}

\draw[->, dashed] (-3,0) -- (3,0) node[anchor=north] {$q$};
\draw	(-0.2,0) node[anchor=north] {0}
		(1,0) node[anchor=north] {1}
		(-1,0) node[anchor=north] {-1};
\draw	(-1,2) node{{\scriptsize LocTrapp}}
		(2,1) node{{\scriptsize Complement of Monot}}
		(-1.5,0.5) node{{\scriptsize Not known}};

\draw[->,dashed] (0,-1.5) -- (0,3) node[anchor=east] {$p$};

\draw
		(0.2,1) node[anchor=north] {1}
		(0.2,-1) node[anchor=north] {-1};

\draw[thick] (-3,1) -- (0,1);
\draw[thick] (-2.5,-1.5) -- (0,1);
\draw[thick] (0,1) -- (2,3);

\end{tikzpicture}
\caption{Phase transition diagram.}
\end{figure}
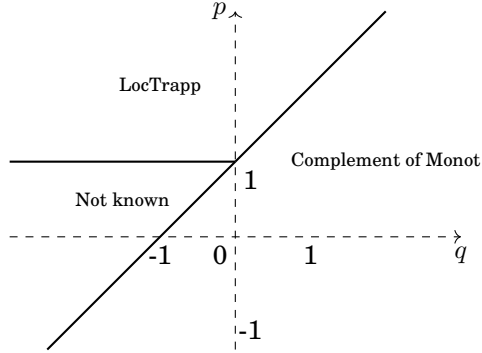
Notice in Figure~2 that for the case $q\leq 0$ we assume $p>1$ to prove that the Local Trapping event happens with probability one.  Furthermore, the Local Trapping event happens with probability one if $q>0 $ and $p>1+q$. In the case $p\leq 1+q$, Theorem~\ref{thm:example2} says that the Monotone event does not happen almost surely.
Observe that Theorem~\ref{thm:example2} does not cover the case $q\leq 0$ and and $1+q< p\leq 1$.

See Theorems~\ref{thm:PMonot0} and~\ref{theo:very-strong} for a complete statement.
\begin{remark}  \rm The above result  implies that trapping in a circuit does not follow alone from the summability of the inverse of $W_1$. By~\cite{cotar2017} for the edge reinforced random walk that would be sufficient.
	Indeed, in the current framework what matters as well is the relation between the reinforcements we impose on edges with positive and negative crossing numbers $c_n(\vec{e})$. One needs to guarantee that the walkers have a bias towards choosing edges with positive crossing numbers. 
\end{remark}

Before reviewing the literature, we present some figures to illustrate Theorem~\ref{thm:example2}.

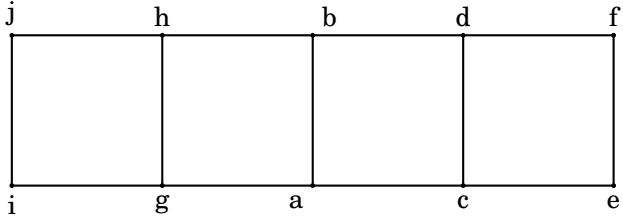
\begin{figure}[h!]
\centering
\begin{tikzpicture}[scale=0.5]
	
	\coordinate (g1) at (0,-2);
	\coordinate (g2) at (0, 2);
	\coordinate (g3) at (4,-2);
	\coordinate (g4) at (4, 2);
	\coordinate (g5) at (8,-2);
	\coordinate (g6) at (8,2);
		\coordinate (l3) at (-4,-2);
	\coordinate (l4) at (-4, 2);
	\coordinate (l5) at (-8,-2);
	\coordinate (l6) at (-8,2);
	\draw[fill] (g1) circle [radius=0.05];
	\node [below left] at (g1) {a};
	\draw[fill] (g2) circle [radius=0.05];
	\node [above right] at (g2) {b};
	\draw[fill] (g3) circle [radius=0.05];
	\node [below] at (g3) {c};
	\draw[fill] (g4) circle [radius=0.05];
	\node [above] at (g4) {d};
	\draw[fill] (g5) circle [radius=0.05];
	\node [below] at (g5) {e};
		\draw[fill] (g6) circle [radius=0.05];
	\node [above] at (g6) {f};
	\draw[fill] (l3) circle [radius=0.05];
	\node [below] at (l3) {g};
	\draw[fill] (l4) circle [radius=0.05];
	\node [above] at (l4) {h};
	\draw[fill] (l5) circle [radius=0.05];
	\node [below] at (l5) {i};
		\draw[fill] (l6) circle [radius=0.05];
	\node [above] at (l6) {j};
	\draw [thick] (g1) -- (g2);
	\draw [thick] (g1) -- (g3);
	\draw [thick] (g2) -- (g4);
	\draw [thick] (g3) -- (g4);
	\draw [thick] (g3) -- (g5);
	\draw [thick] (g4) -- (g6);
	\draw [thick] (g5) -- (g6);
	\draw [thick] (g1) -- (l3);
	\draw [thick] (g2) -- (l4);
	\draw [thick] (l3) -- (l4);
	\draw [thick] (l3) -- (l5);
	\draw [thick] (l4) -- (l6);
	\draw [thick] (l5) -- (l6);
\end{tikzpicture}
\caption{Example of a finite graph.}
\label{fig:1}
\end{figure}

Assume that the graph $G$ is as in Figure~\ref{fig:1} above, and assume that $p>1+q$. Then the above theorem states that the set of walkers split the graph into disjoint subsets in the sense that eventually different groups of walkers spin around disjoint circuits of the graph. For example it may happen that eventually the walkers will be stuck in the thick circuits $(g,i,j,h)$ and $(c,d,f,e)$ in Figure~\ref{fig:4}, and the missing edges will be crossed only a finite number of times.

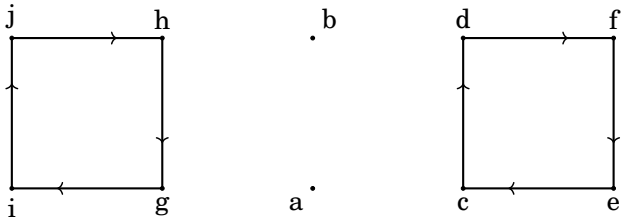
\begin{figure}[h!]
\centering
\begin{tikzpicture}[scale=0.5]
	
	\coordinate (g1) at (0,-2);
	\coordinate (g2) at (0, 2);
	\coordinate (g3) at (4,-2);
	\coordinate (g4) at (4, 2);
	\coordinate (g5) at (8,-2);
	\coordinate (g6) at (8,2);
		\coordinate (l3) at (-4,-2);
	\coordinate (l4) at (-4, 2);
	\coordinate (l5) at (-8,-2);
	\coordinate (l6) at (-8,2);
	\draw[fill] (g1) circle [radius=0.05];
	\node [below left] at (g1) {a};
	\draw[fill] (g2) circle [radius=0.05];
	\node [above right] at (g2) {b};
	\draw[fill] (g3) circle [radius=0.05];
	\node [below] at (g3) {c};
	\draw[fill] (g4) circle [radius=0.05];
	\node [above] at (g4) {d};
	\draw[fill] (g5) circle [radius=0.05];
	\node [below] at (g5) {e};
		\draw[fill] (g6) circle [radius=0.05];
	\node [above] at (g6) {f};
	\draw[fill] (l3) circle [radius=0.05];
	\node [below] at (l3) {g};
	\draw[fill] (l4) circle [radius=0.05];
	\node [above] at (l4) {h};
	\draw[fill] (l5) circle [radius=0.05];
	\node [below] at (l5) {i};
		\draw[fill] (l6) circle [radius=0.05];
	\node [above] at (l6) {j};
	\draw [thick, big arrow 3] (g3) -- (g4);
	\draw [thick, big arrow 3] (g5) -- (g3);
	\draw [thick, big arrow 3] (g4) -- (g6);
	\draw [thick, big arrow 3] (g6) -- (g5);
	\draw [thick, big arrow 3] (l4) -- (l3);
	\draw [thick, big arrow 3] (l3) -- (l5);
	\draw [thick, big arrow 3] (l6) -- (l4);
	\draw [thick, big arrow 3] (l5) -- (l6);
\end{tikzpicture}
\caption{Trapping circuits for $p>1$.}
\label{fig:4}
\end{figure}

The above models fit in the class of random processes with reinforcement. The basic models are urn models and its many variants. We refer the reader to the survey of Pemantle \cite{pemantle2007} for an overview. Here, we will restrict ourselves only to those models that are closely related to ours.

There are manifold motivations to study random processes with reinforcement. One of them is that they can be used to model the behaviour of ants. The article~\cite{ma_xia_yang_2017} from the computer science community, studies an algorithm based on random walks. The algorithm is designed to detect communities on complex networks and uses the idea that ants learn through the pheromone they leave on their path. Regarding the probability community, the authors in~\cite{Kious2020}  propose a model to find a shortest path between two points in a graph through reinforcement learning. Roughly speaking, the ants perform successive random walks in the graph and, when it finds the food, its path is reinforced according to a reinforcement rule. Their methods are based on electrical network techniques for random walks and on Rubin's construction.

\subsubsection*{Literature on non conservative zero range process}

Non conservative versions of the well known zero range process have attracted interest in the physics as well as in the mathematics community. To the best of our knowledge models of similar kind were first introduced in \cite{Arnold1981} with the aim to model chemical reactions.

With the goal to understand condensation phenomena the physics paper~\cite{Angel_2007} considers a model similar to ours.  Based on heuristic arguments they developed  a detailed phase diagram depending on the rates of creation, annihilation, and jump, when the number of sites tends to infinity.

Regarding the mathematics community the article~\cite{Franco2012} is the most relevant for our discussion. Motivated by earlier works, in particular by~\cite{Blount91,Blount92}, the authors consider the model on the discrete one dimensional torus with $N$ vertices. They fix functions $b,d\,:\,\bb R_+\to \bb R_+$ that determine the rate of creation and annihilation, and each particle jumps independently of all other particles. 
Sending $N\to\infty$, the authors established a law of large numbers for the density of particles under an appropriate rescaling.  
In particular if $b(s)-d(s)=s^p$, for $1< p \leq 3$, the limiting equation defined on the one-dimensional torus takes the form
\begin{equation}
\partial_t u = \partial_{xx} u + u^p\,.
\end{equation}
This equation is interesting, since under an appropriate boundary condition it is known that~\cite[Proposition 28.1]{Quittner2007} if $p>3$ the equation has a solution whose $L^1$-norm is finite but whose $L^\infty$ blows up, whereas for $p< 3$ both norms blow up.
In Theorem~\ref{thm:NCZR} we can observe one part of that phenomena that is similar in nature. Indeed, if $p$ is large enough, then there is a unique site on which particles are created and all other sites are at a stand still (corresponding to a blow-up of the $L^\infty$ but not the $L^1$-norm). 

The difference to the present work of course is that the particles do not jump independently of each other (except in the case in which $r=1$), and space is fixed.

\subsubsection*{Comparison with the Edge Reinforced Random Walk.}

The Edge Reinforced Random Walk (ERRW) on undirected graphs is defined in a similar way as the Ant RW. However, for the ERRW one considers undirected edges, and typically only a single walker. More precisely, one assigns a weight $c_n(e)$ to each \textit{undirected} edge $e$, where $c_n(e)$ is the number of times that the walker crossed $e$ irrespective of the direction. In this case only one reinforcement function $W_1$ is needed and the rate to cross an edge $e$ (in any direction) is given by $W_1(c_n(e))$. In particular, if the walker crossed $e$ it increases the chances of crossing it in the opposite direction. This is not the case for the Ant RW. 

The Ant RW shares some similarities with the ERRW on directed graphs but these two processes are not equal. Indeed, as described in~\cite{pemantle2007}, in the ERRW on directed graphs each directed edge $\vec{e}=(u,v)$ has the crossing number $c_n(u,v)$ defined to be as the number of times $1\leq i\leq n$ such that $X_{i-1}=u$ and $X_{i}=v$. Observe that for the Ant RW we subtract from $c_n(u,v)$ the number of times $1\leq i\leq n$ such that $X_{i-1}=v$ and $X_{i}=u$ and we have the relation $c_n(u,v)=-c_n(v,u)$. As a consequence, for the Ant RW the probabilities of successive transitions out a vertex $u$ are affected by the edge the RW used to arrive at $u$. This is not the case for the ERRW on directed edges. Therefore,  even though we use directed edges in our definition of the Ant RW, our model is very different from the ERRW on directed graphs.

The behaviour of the ERRW depends on the assumption one makes about the reinforcement function $W_1$. Two cases are of particular interest. The case of superlinear reinforcement, i.e., $\sum_kW_1(k)^{-1}<\infty$, and the case of linear reinforcement, i.e., $W_1(k)=k$. In the case of superlinear reinforcement the walker localizes on a single edge, i.e., it eventually only jumps across that edge. This phenomena does not occur if the reinforcement is linear. In that case the question of interest is whether the walker is recurrent or transient~\cite{Pemantle1988,merkl}. A feature used to study such a question is that the ERRW with linear reinforcement is partially exchangeable. We point out that our models are not partially exchangeable due the nature of the reinforcement mechanism.

In our model we observe and study different phenomenas. Indeed, under suitable conditions on the relation between $W_1$ and $W_2$, if $W_1$ grows superlinearly localization happens on a circuit, i.e., on at least three edges.

As noted in \cite[Section~5.1]{pemantle2007} for the ERRW either on directed graphs and on trees a direct coupling with a collection of independent P\'olya urns are possible. The behaviour of the ERRW on these classes of graphs are successfully described through this coupling. However, for more general graphs this is not possible. A frequent strategy presented in several works is to implement a comparison of the ERRW with P\'olya urns schemes instead of a direct coupling. A strategy that we also employ as we will explain.

The ERRW with superlinear reinforcement has a long history and it took various stages of refinement to arrive at the full result. We mention here only some selected works. The localization result in~\cite{Limic2003} is derived through a construction of the ERRW similar to Rubin's construction. To implement the comparison with the P\'olya urns scheme the author constructs a suitable super martingale. However,  in~\cite{Limic2003} the reinforcement function is restricted to the form $W_1(k)=k^p$. In~\cite{LimicTarres2007} the authors increased the class of reinforcement functions for which the same localization result holds using a similar strategy. 

The problem was then finally completely solved in the work of Cotar and Tacker~\cite{cotar2017}. The authors prove localization of the ERRW only assuming a superlinear reinforcement. The techniques in~\cite{cotar2017} rely on estimates of the probability to achieve a given configuration for the vector of crossing numbers $(c_n(e))_{e}$ and it does not rely on Rubin's construction. 

The above mentioned methods do not seem to apply directly to our case. The main reason is that for the ERRW the vector of numbers $(c_n(e))_{e}$ are non negative and monotone in time. This is not satisfied in our case and we need one additional step before pursuing a comparison with P\'olya urns schemes. On the other hand, our methods do not seem to apply to the ERRW on undirected graphs. Our technique are designed for models with two or more reinforcement functions in competition in the case that one reinforcement function dominates the others.

\subsubsection*{Main ideas of the proof}
Erhard, Franco, and Reis~\cite{EFR2019} introduced the Ant RW with exponential reinforcement. There it is shown that a single walker eventually localizes in a circuit.  To obtain the result the authors derived a lower bound which is uniform in the environment for the probability that the  walker is trapped in a carefully chosen circuit. No comparison with P\'olya urns was needed. But the method in~\cite{EFR2019} relied heavily on the choice of the exponential reinforcement and does not work for more general reinforcement functions as in Theorem~\ref{thm:example2}. Indeed, for polynomial reinforcement, it can be shown that the probability of being trapped in any circuit is arbitrarily close to zero as a function of the environment. 

Motivated by the mentioned previous works one can pursue a comparison with P\'olya urns schemes. However, at first sight this strategy seems to not apply because, as pointed out before, the vector of crossing numbers for the Ant RW is not monotone and can be negative. The crucial observation is that one can seek a comparison with P\'olya urns schemes conditioned on the event that the vector of crossing numbers is monotone. 

Our main achievement is to present a new way to implement this old strategy. We do that by introducing a reference (or auxiliary) edge reinforced random walk on a directed graph. The reference model is constructed via a P\'olya urn scheme which is easier to analyse. The next step is to compare the original model with the reference model. We do the comparison via the Radon Nikodym derivative of the two processes. We obtain a lower bound on the Radon Nikodym derivative that is uniform on a carefully chosen set of trajectories.  Then we can conclude the argument via Borel Cantelli's Lemma.

The proof for the NCZR then follows by an adaptation of the above arguments.

\subsection{Structure of the paper}
For pedagogical reasons we invert the presentation of the models in the paper. We start with the simpler NCZR and then we study the Ant RW. The outline of the article is as follows:
\begin{itemize}
	\item In Section~\ref{sec:preliminaries} we introduce important notation and the balls-in-bins process.
	\item In Section~\ref{sec:NCZR} we formally introduce the NCZR process and prove Theorem~\ref{thm:NCZR}.
	\item In Section~\ref{sec:AntRW} we formally introduce the Ant RW and fully state  our main results.
	\item In Section~\ref{sec:proofAntRW} we prove the results about the Ant RW.
	\item In the Appendix~\ref{sec:proofexamples} we give examples of classes of functions $W_1$ and $W_2$ that satisfy our various assumptions.
\end{itemize}

\section{Preliminaries}
\label{sec:preliminaries}

In this section we fix some notation, terminology, and we review some results. 

\subsection{Natural numbers}
We write $\bb N=\{0,1,2,\cdots\}$ for the set of natural numbers. For $n \in \bb N$ we define the set $[n]=\{1,\cdots, n\}.$

Let $a,b\in \bb N$ with $a<b.$ We write $\llbracket a,b\rrbracket$ for the discrete interval $\{i\in \bb N\,:\,a\leq i\leq b\}.$ We write $\llbracket a,\infty)$ for the discrete interval $\{i\in \bb N\,:\,a\leq i\}.$

\subsection{Undirected and directed graphs}
We start with graphs and its basic objects. An undirected graph  $G$ is a pair $(V,E)$ where $V$ is a countable set of vertices and $E$ is a set of unordered pairs of $V$.  Given two vertices $v$ and $w$ we write $v\sim w$ if the pair $\{v,w\}$ forms an edge. The degree of $v$ in $G$, $\deg_{G}(v),$ is the number of edges $\{v,w\}$ with $w \in V.$ All graphs considered in this work are locally finite,  connected, without parallel edges and without loops. We associate to $G$ the following quantity
\begin{equation}\label{eq:maxmindeg}
\Delta(G)= \max_{v\in V} \deg_{G}(v),
\end{equation}
which is the maximum degree of $G$. 

A directed graph $D$ is also a pair $(V,\vec{E})$ but in this case $\vec{E}$ is a set of ordered pairs of $V$. We define the outer degree of $v\in V$ via:
\begin{itemize}
	\item  $\deg_{D}^{\rm out}(v)=|\{(v,w)\in \vec{E}\,;\,w\in V\}|.$
\end{itemize}
For an undirected graph $G=(V,E)$ we often identify each non-oriented edge with two oriented edges. More specifically, we define the directed graph $\vec{G}$ where $V(\vec{G})=V(G)$ and $\vec{E}(\vec{G})=\{(v,w),\,(w,v)\,:\,\{v,w\}\in E(G)\}.$ 

\subsection{Circuits and paths}\label{sec:circpath}
We introduce some objects for directed graphs. The definitions for undirected graphs are similar and will be omitted.

Let $D=(V,\vec{E})$ be a directed graph. We say that a collection of vertices $C=(u_0,\ldots,u_{\ell-1})$ is a circuit of length $\ell$ with starting point $u_0$ if \rm{(1)} $u_i\neq u_j$ for $0\leq i\neq j\leq \ell-1$ ,  \rm{(2)} $(u_i,u_{i+1})\in \vec{E}$ for $0\leq i\leq \ell -2$ and if moreover \rm{(3)} $(u_{\ell-1}, u_0)\in \vec{E}$. We also identify two circuits $C$ and $C'$, if one can be obtained from the other by a rotation. This defines an equivalence relation, and any circuit $C$ is identified with its equivalence class. We write $C(D)$ for the total number of circuits in the directed graph $D$.

A path  in $D$ is a map $\pi\,:\,\bb{N}\to V$  such that $(\pi(i),\pi(i+1))\in \vec{E}$, for all $i\in\bb N$.

Let $X$ be a set and $f:\bb{N}\to X$ a function. For $a,b\in\bb N$ and $a < b$ one defines the restriction $f_a^b$ as the map $n\in \llbracket a,b\rrbracket\mapsto f(n)$. Similarly, we write $f_a^\infty$ for the restriction of $f$ to the interval $\llbracket a,\infty)$.

\subsection{Lattice version of a directed graph} 

Given a directed graph $D$ as above, we consider a finite family of random walks on $D$ such that two distinct walkers never jump at the same time.  It is convenient to describe all random walks as a single object. For that reason we introduce the \textit{lattice} version of $D.$

We fix a directed graph $D=(V,\vec{E})$ and a number $N \in \bb N$ of particles. We write $D^N$ for the lattice version of $D$ of dimension $N$ and it is the graph with vertex set
\begin{equation*}
V^N= \{\vec{v}:=(v_1,\ldots, v_N):\, v_1,\ldots, v_N\in V\}\,,
\end{equation*}
and edge set
\begin{equation*}
\begin{split}
\vec{E}^N=\{(\vec v,\vec w)\in V^N\times V^N:&\, \exists\, j\in\{1,2,\ldots,N\} \\
&\text{ s.t. }v_i=w_i\,\forall\, i\neq j \text{ and }(v_j, w_j)\in \vec{E}\}.
\end{split}
\end{equation*}
Observe that the standard lattice $\bb Z^N$ is the lattice version of $\bb Z.$ The collection of random walks will be described as a single (random) trajectory $\pi\,:\,\bb N \longrightarrow V^N.$

\subsection{Markov chains and stopping times}

Let $(Z_n)_{n \in \bb N}$ be a discrete time Markov chain and $(\mc F_n)_{n\in \bb N}$ be its natural filtration.  For $k\in \bb N$ we define the shift operator $\mc S_k$ which associates $(Z_n)_{n \in \bb N}$ with $(Z_n)_{n\geq k}\,,$ i.e. $(\mc S_k\circ  Z)_n= Z_{n+k}$ for all $n\in\bb N$. 

Given a stopping time $\tau$ and $k \in \bb N$ we define the stopping time $\tau\circ \mc S_k$ which is the original stopping $\tau$ seen as a function of $(Z_n)_{n\geq k}.$ The above construction makes also sense if we replace $k$ by another stopping time that is finite almost surely.

\subsection{Balls-in-Bins with feedback}
\label{sec:balls-in-bins}
We introduce an essential example of a reinforcement process for this work.

Let $S$ be a finite set. We consider a vector evolving in time $(\eta_v(n);\,v\in S)_{n\in \bb N}.$ We can interpret $\eta_v(n)$ as the number of balls in bin $v\in S$ at time $n$. In our examples we will choose $S$ as either the set of vertices of a finite graph or the set of neighbours of a fixed vertex. To throw a ball into the bin $v$ (creation of a ball) at time $n$ corresponds to the following transition
\begin{equation*}
\eta_w(n+1)=\begin{cases}
\eta_v(n)+1, &\text{if }w=v,\\
\eta_w(n), &\text{otherwise.}
\end{cases}
\end{equation*}
We fix a feedback (reinforcement) function $f:\bb N\to (0,+\infty).$ Conditioned on  $(\eta_u(n);\,u\in S)$ the above transition occurs with probability
\begin{equation*}
\frac{f(\eta_v(n))}{\sum_{w\in S}f(\eta_w(n))}.
\end{equation*}

We introduce the following events:
\begin{itemize} 
	\item ${\rm Mon}$ is the event that there exists a unique site $w$ and a time $m \in \bb N$ such that $\eta_w(n+1)=\eta_w(n)+1$ and $\eta_v(n+1)=\eta_v(n)$ for $v\neq w$ (monopoly) and for all $n\geq m$.
\end{itemize}    

To simplify notation, given a generic function $g:\bb N \to (0,\infty)$ we write
\begin{equation}
\Sigma(g):=\sum_{k=0}^{\infty}g(k).
\end{equation}
The following result that can be found in~\cite{2005Oliveira} is one of the cornerstones of our proof. After a successful comparison with a P\'olya urn scheme it allows us to conclude the desired result.
\begin{theorem}[Rubin's construction]\label{thm:Rubin}
	For the Balls-in-Bins process with feedback function $f$ the following holds irrespective of the initial conditions:
		 If $\Sigma(f^{-1})<\infty$  then $\bb P({\rm Mon})=1$ (monopolistic regime).
\end{theorem}

\section{Non conservative zero range process}
\label{sec:NCZR}

In this section we provide a precise definition of the NCZR, and prove Theorem~\ref{thm:NCZR}. This should be seen as a warm up for the proof of the corresponding result for the Ant RW.

We fix a finite graph $G=(V,E)$. For a vertex $v\in G$ we define the following operations on elements $\eta=(\eta_w;\,w\in G)\in \bb N^V$:
\begin{itemize}
	\item We define $\eta^{+v}=(\eta_w^{+v}; w\in G)$ via 
	\begin{equation*}
	\eta^{+v}_w\;=\;
	\left\{
	\begin{array}{cl}
	\eta_w\,,& v\neq w\,,\\
	\eta_v+1\,,& v=w\,.\\
	\end{array}
	\right.
	\end{equation*}
	\item If $\eta_v\geq 1$, we define $\eta^{-v}=(\eta_w^{-v}; w\in G)$ via
	\begin{equation*}
	\eta^{-v}_w\;=\;
	\left\{
	\begin{array}{cl}
	\eta_w\,,& v\neq w\,,\\
	\eta_v-1\,,& v=w\,.\\
	\end{array}
	\right.
	\end{equation*}
	\item If $\eta_v\geq 1$ and $u\sim v,$ we define $\eta^{vu}=(\eta_w^{vu}; w\in G)$ via
	\begin{equation*}
	\eta^{vu}_w\;=\;
	\left\{
	\begin{array}{cl}
	\eta_w\,,& w \notin \{v,u\}\,,\\
	\eta_{u}+1\,,& u=w  \,,\\
	\eta_v-1\,,& v=w\,.
	\end{array}
	\right.
	\end{equation*}
\end{itemize}

We proceed to define a Markov chain $H=(H(n);\,n\in \bb N)$ with state space $\bb N^V$. We assume that $H(0)$ is non-zero, and we fix parameters $p,q,r\in \bb R$. The transition probabilities are given by
\begin{equation*}
\begin{split}
\bb P\big[H(n+1)=\eta^{+v}|H(n)=\eta\big]&=\frac{\eta_v^p}{\sum_{w \in V}\big[\eta_w^p+\eta_w^q+\deg(w)\eta_w^r\big]}\,,\\
\bb P\big[H(n+1)=\eta^{-v}|H(n)=\eta\big]&=\frac{\eta_v^q}{\sum_{w \in V}\big[\eta_w^p+\eta_w^q+\deg(w)\eta_w^r\big]}\,,\\
\bb P\big[H(n+1)=\eta^{vu}|H(n)=\eta\big]&=\frac{\eta_v^r}{\sum_{w \in V}\big[\eta_w^p+\eta_w^q+\deg(w)\eta_w^r\big]}\,.
\end{split}
\end{equation*}
Hereafter we adopt the convention that $0^p=1$. This is only necessary to guarantee that our process will not be stuck in the zero configuration. But we still have $0^q=0$ and $0^r=0$.

The event ${\rm Mon}$ is defined as in Section~\ref{sec:balls-in-bins} with $(H(n);\,n\in \bb N)$ instead of $(\eta(n);\,n\in \bb N).$ The rigorous definition of ${\rm Creat}$ is the following: there exists $m\in \bb N$ such that for each $n\geq m$ there exists a vertex $v(n)\in G$ with $H(n+1)=H(n)^{+v(n)}.$

With the above definitions in mind we now prove Theorem~\ref{thm:NCZR}.

\subsubsection*{Proof of Theorem~\ref{thm:NCZR}. }

The strategy of the proof is the following. We construct an auxiliary balls-in-bins process $K^{\tau}$ associated to $H$ and to a finite stopping time $\tau$. Depending on the parameters $p,q,r\in\bb R_+$ it is then possible to compare $K^\tau$ and $H$ in such a way that it allows to deduce the desired result from Theorem~\ref{thm:Rubin}.

Let $\tau$ be a finite stopping time. The auxiliary process $K^{\tau}=(K^{\tau}(n);\,n\geq 0)$ is defined through:
\begin{enumerate}
	\item For $n\in[0,\tau]$ we set $K^{\tau}(n)=H(n).$
	\item For $n > \tau$, $K^{\tau}$ evolves like a balls-in-bins process with feedback function $k\in \bb N\mapsto k^p$, i.e., accordingly to the following probability transitions: for any $v\in G,$
	\begin{equation*}
	\begin{split}
	\bb P\big[K^{\tau}(n+1)=\eta^{+v}|K^{\tau}(n)=\eta\big]&=\frac{\eta_v^p}{\sum_{w \in V}\eta_w^p}\,.
	\end{split}
	\end{equation*}
\end{enumerate}

Our next step is to relate the processes $H$ and $K^{(\tau)}$ via a Radon-Nikodym derivative on the set of trajectories 
\begin{equation}\label{eq:defA+}
\mc{A}^{+}=\{(\eta(m);\,m\in \bb N)\,:\,\forall\,n\in \bb N,\,\eta(n)\in \bb N^V,\,\exists\,v(n)\in V;\,\eta(n+1)= \eta(n)^{+v(n)}\}.
\end{equation}  
In plain words, trajectories in $\mc{A}^{+}$ neither perform walker steps nor annihilate particles. Note that, recalling the notation from Section~\ref{sec:circpath},
\begin{equation}\label{eq:createA+}
{\rm Creat}= \{\eta:\, \exists\, n\in\bb N\, \text{ s.t. } \eta_n^\infty\in\mc A^+\}\,.
\end{equation}

Let $\eta=(\eta(m);\,m\in \bb N)\in \mc{A}^{+}$. For a finite stopping time $\tau$ and $m\in \bb N$ with $m\geq \tau$ we write $\eta_{\tau}^{m}:=(\eta(n);\,\tau\leq n\leq m),$ and with a slight abuse of notation we also write that $\eta_\tau^m\in\mc A^+$. Similar definitions hold for $H$ and $K^{\tau}$. For any $\eta_{\tau}^m\in \mc{A}^{+}$ we define
\begin{equation}\label{def:dHdK}
\frac{dH}{dK}(\eta_{\tau}^m):=\frac{\bb P(H_{\tau}^{m}=\eta_{\tau}^{m})}{\bb P(\big(K^{\tau}\big)_{\tau}^{m}=\eta_{\tau}^{m})}.
\end{equation}

The following Lemma is crucial because it will allow us to compare the original model with a P\'olya urn scheme. 
\begin{lemma}\label{lem:NCZR} In the above setting the following holds irrespective of the choice of $H(0):$ 
	\begin{itemize}
		\item If $p>1+\max\{q,r,0\}$ then there exists $\delta=\delta(|V|,\Delta(G),p,q,r)>0$ such that
		\begin{equation}\label{eq:infdHK}
		\inf\bigg\{\frac{dH}{dK}(\eta_{0}^{m});\,m\in \bb N,\,\eta_{0}^{m}\in \mc{A}^{+}\bigg\}\geq \delta. 
		\end{equation}
		\item If $p\leq 1+\max\{q,r\}$ then for any $\epsilon>0$ there exists $m_0=m_0(\epsilon,p,q,r,\Delta(G))$ such that for all $m\geq m_0$
		\begin{equation}\label{eq:supdHK}
		\sup\bigg\{\frac{dH}{dK}(\eta_{0}^{m});\,\eta_0^m\in \mc{A}^{+}\bigg\}< \epsilon.
		\end{equation}
	\end{itemize}
\end{lemma}

We postpone the proof of Lemma~\ref{lem:NCZR} to Section~\ref{sec:lem:NCZR}. Assuming Lemma~\ref{lem:NCZR} we finish the proof of Theorem~\ref{thm:NCZR}. 

\textit{First case:} assume $p>1+\max\{q,r,0\}.$

We define the following stopping time
\begin{equation}\label{eq:stoppNCZR}
\tau=\inf\{n\geq 0\,:\,H_0^n \notin \mc{A}^{+}\}.
\end{equation}
Note that on the event $\{\tau=\infty\}$ we actually have $H\in {\rm Creat}$. With this in mind our goal is to show that for some shift $\tau'$ of $\tau$ the event $\{\tau'=\infty\}$ holds with probability one with respect to the law of $H.$ 

We define recursively:
\begin{enumerate}
	\item $\tau_{0}=0$,
	\item If $\tau_{k}=\infty$ then  $\tau_{k+1}=\infty$,
	\subitem otherwise, $\tau_{k+1}=\tau\circ \mc{S}_{\tau_{k}}$ (see Section~\ref{sec:preliminaries}).
\end{enumerate}   

We want to show that  $\bb P_{H(0)}\big(\exists\, k\geq 1\,:\,\tau_{k}=\infty\big)=1.$ To that end we will show that $\sum_{k\geq 0}\bb P_{H(0)}(\tau_{k}<\infty)<\infty$ and the result then follows from Borel-Cantelli's lemma. 

Fix $m\in \bb N$. For simplicity, in the calculations below we identify $\big(K^{0}\big)_{0}^m$ with $K_{0}^{m}$ (here $K^0$ refers to the process defined above with $\tau=0$).  In the fourth line of \eqref{eq:step1} below we use Lemma~\ref{lem:NCZR}. Recall that the bound \eqref{eq:infdHK} is uniform in all configurations $\eta_0^m\in \mc{A}^{+}.$ In the fifth line of \eqref{eq:step1} we use that $K^0\in \mc{A}^{+}$ almost surely (cf. Theorem~\ref{thm:Rubin}). Then we obtain
\begin{equation}\label{eq:step1}
\begin{split}
\bb P_{H(0)}(\tau_{1}\geq m)&\geq \bb P_{H(0)}(H_{0}^m\in \mc{A}^{+})\\
&=\sum_{\eta_{0}^m \in \mc{A}^{+}}\bb P_{H(0)}(H_{0}^m=\eta_{0}^m)\\
&=\sum_{\eta_{0}^m \in \mc{A}^{+}}\frac{dH}{dK}(\eta_0^m)\bb P_{H(0)}\big(K_{0}^{m}=\eta_{0}^m\big)\\
&\geq \delta \bb P_{H(0)}\big(K_{0}^{m}\in \mc{A}^{+}\big)\\
&=\delta.
\end{split}
\end{equation}

In particular, sending $m\to \infty$, we obtain that $\bb P_{H(0)}(\tau_1<+\infty)<1-\delta.$ Note that all the estimates above are uniform in the starting configuration, i.e.,
\begin{equation*}
\sup\{\bb P_{\eta}(\tau_1<+\infty);\,\eta \in \bb N^V\}\leq 1-\delta.
\end{equation*} 
As induction hypothesis on $k$ assume that 
\begin{equation}\label{eq:borelcantelli}
\sup\{\bb P_{\eta}(\tau_{k}<+\infty);\,\eta \in \bb N^V\}\leq (1-\delta)^k.
\end{equation} 
Using the strong Markov property we derive
\begin{equation*}
\begin{split}
\bb P_{H(0)}(\tau_{k+1}<\infty)&=\bb P_{H(0)}(\tau_{k+1}<\infty,\,\tau_{k}<\infty)\\
&=\bb E_{H(0)}\big[\one_{\{\tau_{k}<\infty\}}\bb P_{H(\tau_k)}(\tau_1<\infty)\big]\\
&\leq \bb E_{H(0)}\big[\one_{\{\tau_{k}<\infty\}}(1-\delta)\big]\leq (1-\delta)^{k+1}.
\end{split}
\end{equation*}
Hence, by the Borel Cantelli lemma we deduce that
\begin{equation*}
\bb P_{H(0)}\big(\exists\, k\geq 1\,:\,\tau_{k}=\infty\big)=1\,.
\end{equation*}
 However, on the event $\{\exists\, k\in\bb N \text{ s.t. } \tau_k =\infty\}$ we have $H\in {\rm Creat}$. This shows that $\bb P({\rm Creat})=1$. To continue, we define
\begin{equation}\label{eq:leadmon}
\begin{split}
\mc B^{+}&= \left\{\eta\in\mc A^+:\, \exists\, w\in G \text{ s.t. } \forall n\in \bb N:\, \eta_v(n+1) =
\begin{cases}
\eta_v(n)+1,&\text{if }v=w\\
\eta_v(n), &\text{otherwise.} 
\end{cases}
\right\}\,.
\end{split}
\end{equation}
In the same manner as in~\eqref{eq:createA+} it holds that
\begin{equation*}
{\rm Mon}= \{\eta:\, \exists\, n\in\bb N \text{ s.t. }\eta_n^\infty\in \mc B^+\}.
\end{equation*}
Since $p>1$  we can change $\mc{A}^{+}$ by $\mc{B}^+$  in~\eqref{eq:step1} together with Theorem~\ref{thm:Rubin} to deduce that $\bb P({\rm Mon})=1$.

\textit{Second case:} assume that $p\leq 1+\max\{q,r\},$ and let $\epsilon>0$. By Lemma~\ref{lem:NCZR} there exists $m\in \bb N$ such that \eqref{eq:supdHK} holds. Therefore, for any $m\in \bb N$ we have the following bound
\begin{equation}\label{eq:upperNCZR}
\begin{split}
\bb P_{H(0)}(H_0^{m}\in \mc{A}^{+})
&=\sum_{\eta_{0}^m \in \mc{A}^{+}}\bb P_{H(0)}(H_{0}^m=\eta_{0}^m)\\
&=\sum_{\eta_{0}^m \in \mc{A}^{+}}\frac{dH}{dK}(\eta_0^m)\bb P_{H(0)}\big(K_{0}^{m}=\eta_{0}^m\big)\\
&\leq \epsilon \bb P_{H(0)}\big(K_{0}^{m}\in \mc{A}^{+}\big)=\epsilon.
\end{split}
\end{equation}

Therefore, sending $\epsilon\to 0$ we conclude that $
\bb P_{H(0)}(H_0^{\infty}\in \mc{A}^{+})=0.$

Using appropriate stopping times as in the first part of the proof allows us to conclude.
\subsection{Estimating the Radon-Nikodym derivative}
\label{sec:lem:NCZR}
In this section we prove Lemma~\ref{lem:NCZR}.

We start with a brief idea of the proof. The expression $\frac{dH}{dK}$ is a product and we can turn it into a exponential of a sum, using for example the inequality $1+x\leq e^{x}$. The sum obtained in this way can be controlled by using the total number of particles. It then turns out that the new sum will behave like $\sum_{n}1/n^{s}$ where $s=p-\max\{q,r\}$ and the convergence or divergence of this sum will provide the result.

Let $(\eta(n);\,n\in \bb N)\in \mc{A}^{+}$. From definition \eqref{def:dHdK} we obtain the expression 
\begin{equation}\label{eq:exprdHdK}
\begin{split}
\frac{dH}{dK}(\eta_{0}^{m})&=\prod_{k=0}^{m-1}\frac{\sum_{w \in G}\eta_{w}(k)^p}{\sum_{w \in G}[\eta_{w}(k)^{p}+\eta_{w}(k)^{q}+\deg(w)\sigma_{w}(k)^r]}\\
&=\prod_{k=0}^{m-1}\frac{1}{1+\frac{\sum_{w}\eta_{w}(k)^{q}}{\sum_{w}\eta_{w}(k)^p}+\frac{\sum_{w}\deg(w)\eta_{w}(k)^{r}}{\sum_{w}\eta_{w}(k)^p}}.
\end{split}
\end{equation}

Define $N_{\eta}(k):=\sum_{w\in G}\eta_w(k)$, which is the total number of particles of the configuration $\eta(k)\in \bb N^V.$ Observe that there exists a constant $C$ depending on $|V|$, $\Delta(G)$, $p$, $q$, and $r$ such that\begin{equation}\label{eq:boundpqr}
\begin{split}
C^{-1}N_{\eta}(k)^s&\leq \sum_{w}\eta_{w}(k)^{s}\leq C(N_{\eta}(k)^s\one_{s\geq 0} +\one_{s<0}),\,\text{ for } s=p,\,q,\,\text{ and }\\
C^{-1}N_{\eta}(k)^r&\leq \sum_{w}\deg(w)\eta_{w}(k)^{r}\leq C(N_{\eta}(k)^r \one_{r\geq 0} + \one_{r<0})\,.
\end{split}
\end{equation}

\textit{First case:} assume that $p>1+\max\{q,r,0\}.$
From the inequality $1/(1+x)\geq e^{-x}$, \eqref{eq:exprdHdK}, and \eqref{eq:boundpqr} we obtain that
\begin{equation}\label{eq:finalLB}
\begin{split}
\frac{dH}{dK}(\eta_{0}^{m})&\geq \exp\bigg(-\sum_{k=0}^{m-1}\frac{\sum_{w}\eta_{w}(k)^{q}}{\sum_{w}\eta_{w}(k)^p}-\sum_{k=0}^{m}\frac{\sum_{w}\deg(w)\eta_{w}(k)^{r}}{\sum_{w}\eta_{w}(k)^p}\bigg)\,.
\end{split}
\end{equation}
Applying~\eqref{eq:boundpqr} we see that
\begin{equation}
\sum_{k=0}^{m-1}\frac{\sum_{w}\eta_{w}(k)^{q}}{\sum_{w}\eta_{w}(k)^p} \leq \sum_{k=0}^{m-1}C^2\frac{N_\eta(k)^q \one_{q\geq 0} + \one_{q < 0}}{N_\eta(k)^p}\,
\end{equation}
and a similar estimate holds for the second term in the exponential in~\eqref{eq:finalLB}.
Observe that the total number of particles $N_{\eta}(k)$ increases by one unit at each time because $\eta \in \mc{A}^{+}$. Thus, the condition on $p$ is enough to conclude~\eqref{eq:infdHK}.

\textit{Second case:} assume that $ p\leq 1+\max\{q,r\}$. We will assume that $\max\{q,r\}=q$ because the proof for the other case is similar. We first treat the case that $\max\{q,r\}\leq p$.  Using \eqref{eq:boundpqr} in \eqref{eq:exprdHdK} we obtain that
\begin{equation*}
\frac{dH}{dK}(\eta_{0}^{m})\leq \prod_{k=0}^{m-1}\frac{1}{1+C^{-2}N_{\eta}(k)^{q-p}}.
\end{equation*}
Observe that for $x\in (0,2]$ it holds that $1/(1+x)\leq e^{-x/2}.$ Since $q-p\leq 0$ we have that $C^{-2}N_{\eta}(k)^{q-p}\leq 2$ (we can enlarge $C$ if needed). Therefore,
\begin{equation*}
\frac{dH}{dK}(\eta_{0}^{m})\leq \exp\bigg(-\frac{1}{2}\sum_{k=0}^{m-1}C^{-2}N_{\eta}(k)^{q-p}\bigg).
\end{equation*}
The sum on the RHS diverges since $p-q\leq 1,$ and as consequence the bound goes to $0$ as $m\to \infty.$
If on the other hand $p < q= \max\{q,r\}$, then we estimate
\begin{equation*}
\begin{split}
\prod_{k=0}^{m-1}\frac{1}{1+C^{-2}N_{\eta}(k)^{q-p}}
&= \exp\left\{-\sum_{k=0}^{m-1}\log(1+C^{-2}N_\eta(k)^{q-p})\right\}\\
&\leq \exp\left\{-\sum_{k=0}^{m-1}\log(C^{-2}N_\eta(k)^{q-p})\right\}\\
&=\exp\left\{-\sum_{k=0}^{m-1}\Big(\log(C^{-2}) + (q-p)\log(N_\eta(k))\Big)\right\}\,.
\end{split}
\end{equation*}
Again using that $k\mapsto N_\eta(k)$ is strictly monotone increasing, we see that the right hand side above converges to zero when $m$ tends to infinity. We therefore can conclude.

\section{The Ant random walk}
\label{sec:AntRW}
In this section we give a precise definition of the Ant RW. We end the section with the full statement of Theorem~\ref{thm:example2}.

We fix a graph $G=(V,E)$ and we assume that the graph is not a tree, i.e., it has at least one circuit. We further fix $N\in \bb N$, the total number of walkers.
\subsubsection*{Definition of the model.} The Ant RW is a process in discrete time
\begin{equation*}
(\vec{X}(n);\,n\in \bb N)=(X_i(n);\,i \in [N]\,,n\in \bb N)
\end{equation*}
which takes values in $V^N.$ For $n\geq 0,$ let $\mc F_n=\sigma(\vec{X}(m);\,0\leq m\leq n)$ be the natural filtration associated to the process $\vec X$. 

For any $(u,v)\in \vec{E}$ and $n\geq 0$, let
\begin{equation*}
\begin{split}
c^{(i)}_{n}(u,v):=&\sum_{m=1}^{n} \big(\one_{[(X_i(m-1),X_i(m))=(u,v)]}-\one_{[(X_i(m-1),X_i(m))=(v,u)]}\big).
\end{split}
\end{equation*}
In plain  words, $c^{(i)}_{n}(u,v)$ is  the number of times the walker $X_i$ jumped from $u$ to $v$ minus the number of times $X_i$ jumped from $v$ to $u$ in the time interval $\llbracket 0, n\rrbracket$. The (collective) crossing  number of the edge $(u,v)$ at time $n$ is defined by
\begin{equation*}
c_{n}(u,v)=\sum_{i=1}^{N}c^{(i)}_{n}(u,v).
\end{equation*}
Observe that at time $0$, $c_{0}(u,v)=0$ for all $(u,v)\in \vec{E}.$ 

We further fix two reinforcement functions $W_1,W_2:\bb N\to (0,+\infty)$, and we define $W:\bb Z\to (0,+\infty)$ via
\begin{equation*}
W(k) = \begin{cases}
W_1(k), &\text{if } k >0,\\
W_2(-k), &\text{if } k \leq 0,
\end{cases}
\end{equation*} 
and we emphasize that we assume $W_i(k)\neq 0$ for any $k\in \bb{N}$, $i=1,2$.

The dynamics of the Ant RW is described as follows. We choose an arbitrary (but fixed) distribution for the starting point $\vec{X}(0)  \in V^N$. Assume that the process is defined in the time interval $\llbracket 0, n-1\rrbracket$. At time $n$ we select an index $i\in[N]$ uniformly at random, independent of the past. Independent of the selection of $i$, conditioned on $\mc{F}_n$,  one has that $X_i(n)=u$ with probability 
\begin{align} \label{eq:transitions}
\frac{W(c_{n-1}(X_i(n-1),u))}{\displaystyle
	\sum_{ w:w\sim  X_i(n-1)} W(c_{n-1}(X_i(n-1),u))}.
\end{align}

Observe that the joint process $(\vec{X}(n),c_n(\cdot,\cdot);\,n\geq 0)$ is a discrete time Markov chain. We denote by $\bb P$ the law of the above process when started from the configuration $(\vec{X}(0),c_0(\cdot,\cdot))$.  In the future it will be convenient to condition the process on a given configuration $\xi=(\vec{X}(m),c_m(\cdot,\cdot))$, for a possible random time $m,$ and start afresh from there. In that case we write $\bb P_{\xi}$ for the law of the new process.


\subsubsection*{Long time behaviour}
We introduce more objects in order to fully state our main result. 

Let $V^{\infty}$ be the set of vertices $v\in V$ that are visited infinitely many  times by $\vec{X}$. Let $G^{\infty}=(V^{\infty},E^{\infty})$ be the induced graph by $V^{\infty}$, i.e., $\{u,v\}\in E^{\infty}$ if $\{u,v\}\in E$ and $u,v\in V^{\infty}$.

Fix a circuit $C$, and a representative $(u_0,u_1,\cdots,u_{\ell-1})$ of $C$ (recall from Section~\ref{sec:preliminaries} that $C$ actually denotes an equivalence class).

We introduce two events:
\begin{itemize}
	\item ${\rm Monot}$ is the event that there exists $m\in \bb N$ such that  $n\geq m\mapsto c_n(u,v)$ is a monotone function, for all $(u,v)\in \vec{E}$.

	\item ${\rm LocTrapp}$ is the event that there exists $m\in \bb N$, such that
	\subitem\rm{a)} the vector of crossing numbers $(c_n(u,v);\,n\geq m,\,(u,v)\in \vec{E})$ is monotone,
	\subitem\rm{b)} there exists $K\leq N$ and disjoint circuits $(C_i;\,i\in [K])$ such that
	\begin{equation*}
	V^{\infty}=\bigcup_{i=1}^{K}V(C_i)\,\,\,,\,\,
		E^{\infty}=\bigcup_{i=1}^{K}E(C_i)\,.
		\end{equation*}

\end{itemize}

 As an immediate consequence of the definition of {\rm LocTrapp}, the walkers are trapped in the circuits $C_i$ in the sense that they are not allowed to visit any other edges for large times. Furthermore, since the vector of crossing numbers is monotone, a group of walkers trapped in the circuit $C_i$ are forced to spin in the same direction.

Before proceeding we make some remarks about the above events. In particular, we compare them with their counterpart in the context of the balls-in-bins process (cf. Section~\ref{sec:balls-in-bins}).

 The event {\rm LocTrapp} resembles the event  {\rm Mon}.  Here the analogy is between creation of particles and crossing circuits. We observe also that {\rm LocTrapp} is a generalization of the trapping event one finds in \cite{EFR2019}.

Now we introduce the assumptions we will make about the reinforcement functions. Define $\mc P_{\leq}$ as the class of functions $f:\bb N\to (0,\infty)$ with the following property: for every $i\in \bb N$ there exists a constant $C_i=C_i(f)\in \bb R_+$ such that
\begin{equation}\label{eq:defF}
\begin{split}
\forall\,\, (a_1,\cdots,a_i)\in \bb N^{i}\,\,\,,\,\, \sum_{j=1}^if(a_j)\leq C_if\bigg(\sum_{j=1}^ia_j\bigg).
\end{split}
\end{equation}
Similarly, we define $\mc{P}_{\geq}$ as the class of functions that satisfy~\eqref{eq:defF} with a lower bound instead of upper bound.
For the next result we further assume that $W_2$ satisfies:
\begin{equation}\label{eq:extraW2}
\exists\, C=C(W_2)\,:\,\forall\,M\in \llbracket-N,+N\rrbracket,\,W_2(n+M)\leq CW_2(n),\,\forall\,n\in \bb N.
\end{equation} 
For $g:\bb N\to (0,+\infty)$ we write $\Sigma(g)=\sum_{n\geq 0}g(n)$.

\begin{theorem}\label{thm:PMonot0} Assume that $W_1\in \mc{P}_{\leq}$, $W_2\in \mc{P}_{\geq}$, $W_2$ satisfies \eqref{eq:extraW2}, and that $ \Sigma(W_2/W_1)=\infty$, then
\begin{equation*}
\bb P({\rm Monot})=0\,.
\end{equation*}
\end{theorem}

We prove Theorem~\ref{thm:PMonot0} in Section~\ref{sec:proofAntRW}. We show in Proposition~\ref{prop:examples} examples of functions in the classes $\mc P_{\geq}$ and $\mc P_{\leq}$ and of those satisfying~\eqref{eq:extraW2}. 

Next we will state our main theorem for the regime $\Sigma(W_2/W_1)<\infty$. In this case our result holds under a more general condition on the pair $(W_1,W_2)$. If in particular, $W_1\in\mc{P}_{\geq}$ and $W_2\in\mc{P}_{\leq}$, then this condition will be satisfied.

We proceed to state the general condition. Given a vector $\phi\in\bb Z^d$ we define the two quantities
\begin{equation}\label{eq:Upsilon}
\Upsilon^{+}(\phi) = \sum_{i: \phi(i) >0} W_1(\phi(i)),\quad \text{and}\quad 
\Upsilon^{-}(\phi) = \sum_{i: \phi(i) \leq 0} W_2(-\phi(i)).
\end{equation}

These quantities are related to our model due the following observation. Assume that at a given time $n$, the walker $X_k$ is on a vertex $v$ of degree $d$ and that $(c_n(v,w))_{w\sim v}= (\phi_n(i))_{i\in[d]} $. Then, the probability of $X_k$ to walk over an edge with a positive crossing number in the next step is precisely given by
\begin{equation}\label{eq:Upsilonratio}
\frac{1}{1+\frac{\Upsilon^{-}(\phi_n)}{\Upsilon^{+}(\phi_n)}}.
\end{equation}
It is therefore natural to impose a control on the ratio $\frac{\Upsilon^{-}(\phi_n)}{\Upsilon^{+}(\phi_n)}$. 

As it turns out we do not need to control all possible vector functions of the form $\{(\phi_n)_{n\geq 1}:\, \phi_n\in \bb Z^d\}$ but only those that arise from a vector of crossing numbers. 



Recall the definition of $\Delta=\Delta(G)$ in \eqref{eq:maxmindeg}.
\begin{definition}\label{def:Phi}
	We denote by $\Phi$ the set of all functions $(\phi_n)_{n\geq 0}$ satisfying
	\begin{itemize}
		\item[(1)] There is a $d\in [1, \Delta]$  such that for all $n\geq 0$ we have that $\phi_n\in \bb Z^d$.  
		\item[(2)] 
		\begin{equation*}
		\sum_{j=1}^{d} \phi_n(j) \in \llbracket-N,+N\rrbracket \qquad (\mathrm{flow\, property})\,.
		\end{equation*}
		\item[(3)] For any time $n\geq 0$ there are $j_1,j_2\in[d]$ such that $\phi_{n}(j_1)>0$, and $\phi_{n+1}(j_1)=\phi_{n}(j_1)+1$ as well as $\phi_{n}(j_2)\leq 0$ and $\phi_{n+1}(j_2)= \phi_{n}(j_2)-1$. Moreover, for all $j\notin\{j_1,j_2\}$ one has that $\phi_{n+1}(j) = \phi_{n}(j)$.
	\end{itemize}
\end{definition}
Here, the second item reflects Equation~\eqref{def:Totalflow}.  The last item reflects the fact that exiting a vertex $v$ through the edge $(v,w)$ increases $c_n(v,w)$ by one, whereas entering the vertex $v$ through the edge $(w,v)$ diminishes $c_n(v,w)$ by one, and the crossing numbers associated to all other edges going out of $v$ remain constant. 

Our general condition on the pair $(W_1,W_2)$ is that
\begin{equation}\label{eq:A2}
\sup_{\phi\in\Phi}\sum_{n=0}^{+\infty}  \frac{\Upsilon^{-}(\phi_n)}{\Upsilon^{+}(\phi_n)}<\infty.
\end{equation}

At first sight it seems difficult to verify that \eqref{eq:A2} holds in practice. However, in Proposition~\ref{prop:examples}, we show that in many cases of interest~\eqref{eq:A2} actually is satisfied.

\begin{theorem}\label{theo:very-strong} Assume that \eqref{eq:A2} holds. Without further assumptions on $W_1$ and $W_2$ we have that $\bb P({\rm Monot})=1$. As a consequence of \eqref{eq:A2}  we have that  
	 $\Sigma(W_1^{-1})<\infty$ and $\bb P({\rm LocTrapp})=1$.
\end{theorem}

The next result gives examples of functions $W_1$ and $W_2$ that satisfy Assumptions~\eqref{eq:defF},~\eqref{eq:extraW2} and~\eqref{eq:A2}.
\begin{proposition}\label{prop:examples} The following holds:
	\begin{itemize}
		\item[\rm{(a)}] Assume that $\Sigma(W_1^{-1})< \infty$, and that $W_2$ is non-increasing, then $(W_1,W_2)$ satisfies~\eqref{eq:A2}.
		\item[\rm{(b)}]  Let $p, q \in \bb R_+.$ Define $W_1, W_2:\bb N\to \bb R_+$ via
		\begin{equation*}
		W_i(k)=\begin{cases}k^p+1 & \text{if }i = 1,\\
		k^q+1 &\text{if } i= 2.\end{cases}\,
		\end{equation*}
		Then $W_1$ and $W_2$ belong to $\mc{P}_{\leq}\cap \mc{P}_{\geq}$ and satisfy~\eqref{eq:extraW2}. If $p-q>1$, then $(W_1,W_2)$ satisfies~\eqref{eq:A2}. Moreover, the same conclusions hold if $W_1$ is a polynomial of degree $p$ and $W_2$ is a polynomial of degree $q$. 
		\item[\rm{(c)}]For fixed $p,q\in \bb R_+$ define $\log^+ k= log\, k\vee 1$, and let
		\begin{equation*}
		W_1(k) = k (\log^+ k)^{p}+1,\quad\text{and}\quad W_2(k) = (\log^+ k)^{q} \,.
		\end{equation*}
		If $p-q>1$, then $(W_1,W_2)$ satisfies~\eqref{eq:A2}. 
		\item[\rm{(d)}] If $W_1\in \mc{P}_{\geq}$, $W_2\in \mc{P}_{\leq}$, $W_2$ satisfies \eqref{eq:extraW2}, and $ \Sigma(W_2/W_1)<\infty$, then \eqref{eq:A2} holds.
	\item[\rm{(e)}] If $W_1(k)=e^{\beta k}$ and $W_2(k)\leq e^{\alpha k}$ with $\alpha,\beta \in (0,\infty)$ and $\beta>\Delta(G)\alpha$. In this case the pair $(W_1,W_2)$ satisfies \eqref{eq:A2}.
	\end{itemize}
\end{proposition}
\begin{remark} \rm We expect that Theorem~\ref{theo:very-strong} holds for the functions in item (e) of Proposition~\ref{prop:examples}  with $\beta$ and $\alpha$ only satisfying $\beta>\alpha$. 
\end{remark}
We present the proof of Proposition~\ref{prop:examples} in Appendix~\ref{sec:proofexamples}.

\subsection{A remark about the result for an infinite graph}

Let $G$ be an infinite graph. The goal of this short section is to explain which assumptions one additionally needs to impose on $G$ in order to extend our main results, by combining the results of the current articles with the techniques of~\cite{EFR2019}.

Firstly, one needs to impose that $G$ has a bounded degree.

Secondly, one needs to assume that $G$ is sufficiently connected and that it has sufficiently many circuits. More precisely, one needs to impose that there exists a sequence $(B_n)_{n \in \bb N}$ of subgraphs of $G$ such that
\begin{enumerate}
	\item $B_n\subset B_{n+1},$
	\item $G=\cup_{n\in \bb N}B_n$,
	\item $B_{n+1}\setminus B_n$ is connected and it has at least one circuit.
\end{enumerate} 
Call $\mc{G}$ the class of graphs satisfying the above assumptions. It is simple to check that $\bb Z^d\in \mc{G}$ with $d\geq 2$ and the usual graph structure.

Then the following is true:
\begin{center}
\textit{Assume that $G\in \mc G$, \eqref{eq:A2} holds  and $\Sigma(W_1^{-1})<+\infty.$\\ Then $\bb P({\rm LocTrapp})=1.$
}
\end{center}

\section{Proofs for the Ant RW}
\label{sec:proofAntRW}

The strategy of the proof is identical to the proof of Theorem~\ref{thm:NCZR} (cf. Section~\ref{sec:NCZR}). We construct an auxiliary multi particle random walk $\vec{Y}^\tau$ associated to $\vec{X}$ and to a finite stopping time $\tau$. Depending on the choices of the reinforcement functions $W_1$ and $W_2$ it is possible to compare $\vec{Y}^\tau$ and $\vec{X}$ in such a way that both processes have the same support. An analogue of Theorem~\ref{thm:Rubin} will provide a description of the trajectories of $\vec{Y}^{\tau}$ and from this we deduce the desired result.

Our first step in the construction of $\vec{Y}^{\tau}$ is to define a suitable oriented graph $D_{\tau}=(V_{\tau},\vec{E}_{\tau})$ such that $(v,w)\in \vec{E}_{\tau}$ if and only if $c_{\tau}(v,w)>0$. Eventually  $\vec{Y}^{\tau}$ will take values in $D_{\tau}$.

For each $n\in \bb N$ we define the set of end points of $\vec{X}$ at time $n$ as
\begin{equation}\label{eq:defEnd}
{\rm End}_n=\{X_j(n)\,:\,j\in [N]\}\,.
\end{equation}
Motivated by the flow property following Equation~\eqref{def:Totalflow}, which we will introduce in the forthcoming Section~\ref{sec:lem:taufinite}, we define the following set of vertices for each time $n\in \bb N:$
\begin{equation}\label{eq:defS}
S_n=\{v\in V\,:\,\exists\,w\in V\,\text{ s.t. }|c_n(v,w)|\geq N+1\}\,.
\end{equation}
Let $\tau$ be the first time such that ${\rm End}_n\subset S_n.$ Further, let $\mc{T}_{f}$ be the collection of stopping times that are finite almost surely with respect to the law of $\vec{X}$. For each $\rho\in \mc{T}_{f}$ we write $\tau_{\rho}=\tau\circ\mc{S}_{\rho},$ where we recall that $\mc{S}$ denotes the shift operator.

Provided that $\tau$ is finite we define $D_{\tau}=(V_{\tau},\vec{E}_{\tau})$ as the directed graph given by:
\begin{equation}\label{def:D}
\begin{split}
V_{\tau}=& {\rm End}_{\tau}\cup\{v\in V:\,v \notin {\rm End}_{\tau},\text{ and }\exists\,w\sim_G v\,;\,c_{\tau}(v,w)\neq 0\} \,,\\
\vec{E}_{\tau}=&\{(v,w)\in V_{\tau}^2:\,c_{\tau}(v,w)>0\}\,.
\end{split}
\end{equation}
If $(v,w)\in\vec{E}_{\tau}$, then we also write $v \sim_{D_{\tau}} w$. 
Observe that if $\tau$ is finite, then by definition $X_j(\tau)\in S_{\tau}$ for all $j\in [N].$

\begin{lemma}\label{lem:taufinite} Assume that $G$ has at least one circuit. Let $\rho \in \mc{T}_f.$ Then $\tau_{\rho}$ is finite with probability one with respect to the law of $\vec{X}.$ Furthermore, $\deg_{D_{\tau_{\rho}}}^{\rm out}(v)\geq 1 \,\forall\,v\in V_{\tau}$. In particular, each connected component of $D_{\tau_{\rho}}$ has at least one circuit.
\end{lemma}

We prove Lemma~\ref{lem:taufinite} in Section~\ref{sec:lem:taufinite}. We proceed to define $\vec{Y}^{\tau}$ assuming the validity of the above result. The same construction can also be done with $\tau_\rho$ in place of $\tau$.  Here we write $c_\tau$ for the vector of crossing numbers associated to $\vec{Y}^\tau$. In particular, $c_{\tau,n}(v,u)$ is the crossing number of the edge $(v,u)$ at time $n$ associated to the process $\vec{Y}^\tau$.

The auxiliary process $\vec{Y}^{\tau}=(\vec{Y}^{\tau}(n);\,n\geq 0)$ is defined via:
\begin{enumerate}
	\item For $n\in [0,\tau]$ we set $\vec{Y}^{\tau}(n)=\vec{X}(n)$ and $c_{\tau,n}(\cdot,\cdot)=c_n(\cdot,\cdot).$
	\item For $n\geq \tau,$ $\vec{Y}^{\tau}$ evolves like a multi particle reinforced random walk on a directed graph with reinforcement function $W_1$: 
	\subitem{\rm a)} We select an index $j\in [N]$ uniformly at random.
	\subitem{\rm b)} Conditioned on $Y^{\tau}_j(n)=v\in D_\tau$ and $\mc F_n$, we set $Y^\tau_j(n+1)=w\in D_\tau$ with probability 
	\begin{equation*}
	\frac{W_1(c_{\tau,n}(v,w))}{\displaystyle\sum_{u:v\sim_{D_{\tau}} u}W_1(c_{\tau,n}(v,u))}\one\{v\sim_{D_{\tau}}w\}.
	\end{equation*} 
	
	\subitem{\rm c)} If $(Y^\tau_j(n),Y^\tau_j(n+1))=(v,w)$ we update the crossing number as follows: 
	\begin{equation*}
	c_{\tau,n+1}(\bar{v},\bar{w})=\begin{cases}c_{\tau,n}(\bar{v},\bar{w})+1,&\,\text{if }(\bar{v},\bar{w})=(v,w),\\
	c_{\tau,n}(\bar{v},\bar{w}),&\,\text{otherwise}.\end{cases}
	\end{equation*}
	
\end{enumerate}
\begin{remark}\label{rem:alltimes}\rm Since $\deg_{D_\tau}^{\rm out}(v)\geq 1$ for all $v\in D_{\tau}$ by Lemma~\ref{lem:taufinite} we have that  $\vec{Y}^{\tau}$ is defined for all times. 
\end{remark}

As in Section~\ref{sec:NCZR} we relate the processes $\vec{X}$ and $\vec{Y}^{\tau}$ via a Radon-Nikodym derivative on a special set of trajectories. Recall the lattice version $D^{N}_{\tau}$ defined in Section~\ref{sec:preliminaries} and that a path $\pi:\llbracket \tau,+\infty )\to V^{N}_{\tau}$ represents a possible trajectory of $\vec{Y}^{\tau}.$ The special set of trajectories is given by
\begin{equation}\label{def:Atau}
\mc{A}^{\uparrow}_{\tau}=\{\pi_{\tau}^{\infty}:=\pi:\llbracket \tau,+\infty )\to V^{N}_{\tau};\,\pi_{\tau}^{\infty}\text{ is a trajectory }\}\,.
\end{equation}
Observe that if $\pi_{\tau}^{\infty}\in \mc{A}_{\tau}^{\uparrow}$, on the event $\{(\vec{Y}^{\tau})_{\tau}^{\infty}=\pi_{\tau}^{\infty}\}$ the crossing numbers $n\geq \tau\mapsto |c_{\tau,n}(\cdot,\cdot)|$ are non-decreasing. 

Now we are ready to define the Radon-Nikodym derivative. To that end recall the notation introduced in Section~\ref{sec:circpath}. To simplify notation we identify $(\vec{Y}^{\tau})_{\tau}^{m}$ with $\vec{Y}_{\tau}^{m}$. For any $\pi_{\tau}\in \mc{A}_{\tau}^{\uparrow}$ and any $m\geq \tau$ we define
\begin{equation}\label{def:dXdY}
\frac{d\vec{X}}{d\vec{Y}}(\pi_{\tau}^{m})=\frac{\bb P(\vec{X}_{\tau}^{m}=\pi_{\tau}^{m})}{\bb P(\vec{Y}_{\tau}^{m}=\pi_{\tau}^{m})}\,.
\end{equation}

The following Lemma is the central piece of our results. Its proof can be found in Section~\ref{sec:lem:AntRW}.
\begin{lemma}\label{lem:AntRW}Assume that the graph $G$ has at least one circuit. The following holds almost surely:
	\begin{itemize}
		\item Assume that $W_1\in \mc{P}_{\leq}$, $W_2\in \mc{P}_{\geq}$, $W_2$ satisfies \eqref{eq:extraW2}, and that $ \Sigma(W_2/W_1)=\infty$. Then for any $\epsilon>0$ there exists $m_0=m_0(\epsilon,W_1,W_2,\Delta(G))$ such that for all $m\geq m_0$
		\begin{equation}\label{eq:upperboundXYeps}
		\sup\bigg\{\frac{d\vec{X}}{d\vec{Y}}(\pi_{\tau}^{m});\,\pi_{\tau}^{m}\in \mc{A}_{\tau}^{\uparrow}\bigg\}<\epsilon\,.
		\end{equation}
		\item Assume that~\eqref{eq:A2} holds for the pair $(W_1,W_2)$. Then
		\begin{equation}\label{eq:lowerbounddXdY}
		\exists\,\delta=\delta(|V|,\Delta(G),W_1,W_2)\,:\,\inf\bigg\{\frac{d\vec{X}}{d\vec{Y}}(\pi_{\tau}^{m});\,m\in \bb N,\,\pi_{\tau}^{m}\in \mc{A}_{\tau}^{\uparrow}\bigg\}\geq \delta\,.
		\end{equation}
	\end{itemize}
\end{lemma}
\begin{remark} \rm Lemma~\ref{lem:AntRW} remains unchanged if we substitute $\tau$ by $\tau\circ\mc{S}_{\rho}$ for any finite stopping time $\rho$.
\end{remark}

Next, we describe the support of the process $\vec{Y}^{\tau}.$ 

To that end, we define the directed graph $\mc{D}_{\tau}^{\infty}=(\mc{V}_{\tau}^{\infty},\vec{\mc{E}}_{\tau}^{\infty})$ via 
\begin{equation}\label{def:Dinfty}
\begin{split}
\mc{V}_{\tau}^{\infty}&=\{v\in V_{\tau}\,:\,v \text{ is visited infinitely often by }\vec{Y}^{\tau}\}\,,\\
\vec{\mc{E}}_{\tau}^{\infty}&=\{(v,w)\in (\mc{V}_{\tau}^{\infty})^2\,:\,(v,w)\in \vec{E}_{\tau}\}\,.
\end{split}
\end{equation}
By Remark~\ref{rem:alltimes} we have that $\mc{D}_{\tau}^{\infty}\neq \emptyset.$

We define the following event:
\begin{itemize}
	\item  $\{\vec{Y}^\tau\in \mc{C}^{\uparrow}_{\tau}\}$ is the event that $\vec{Y}^\tau\in \mc{A}_{\tau}^{\uparrow}$ and such that there exist $m\geq \tau$, and $K\leq N$, and disjoint circuits $(C_i;\,i\in [K])$ in $D_{\tau}$ satisfying
	\begin{equation*}
	\bigcup_{i=1}^{K} C_i=\mc{D}_{\tau}^{\infty}\,.
	\end{equation*}
	Here, the last identity has to be understood in the sense that the edge and vertex sets of both graphs coincide.
\end{itemize}

As a consequence of the definitions we obtain that
\begin{equation*}
\begin{split}
{\rm Monot}&= \{\pi\,:\,\exists\,\rho\in \mc{T}_{f} \text{ s.t. }\pi_{\tau_{\rho}}^{\infty}\in \mc{A}^{\uparrow}_{\tau_{\rho}}\}\,,\\
{\rm LocTrapp}&=\{\pi\,:\,\exists\, \rho\in \mc{T}_{f} \text{ s.t. }\pi_{\tau_{\rho}}^{\infty}\in \mc{C}^{\uparrow}_{\tau_{\rho}}\}\,.
\end{split}
\end{equation*}

\begin{lemma}\label{lemma:suppYn} Let $\rho \in \mc{T}_{f}$ and $\tau_{\rho}=\tau\circ\mc{S}_{\rho}.$ If $G$ has at least one circuit then for any choice of $W_1$, we have that $\bb P(\vec{Y}^{\tau_{\rho}}\in \mc{A}^{\uparrow}_{\tau_{\rho}})=1.$ Furthermore, if $\Sigma(W_1^{-1})<\infty$, then $\bb P(\vec{Y}^{\tau_{\rho}}\in \mc{C}^{\uparrow}_{\tau_{\rho}})=1.$

	If the graph $G$ has only one circuit, then 
	\begin{equation}
	\{\vec{Y}^{\tau_{\rho}}\in \mc{A}^{\uparrow}_{\tau_{\rho}}\}=\{\vec{Y}^{\tau_{\rho}}\in \mc{C}^{\uparrow}_{\tau_{\rho}}\}\,,
	\end{equation}
	and in particular, the above events occur with probability one.
\end{lemma}
The above result resembles the phenomena observed in Theorem~\ref{thm:Rubin}. The proof of Lemma~\ref{lemma:suppYn} is also based on Theorem~\ref{thm:Rubin} and the reader will find it in Section~\ref{sec:lemma:suppYn}.

Assuming  Lemmas~\ref{lem:taufinite},\,\ref{lem:AntRW} and~\ref{lemma:suppYn} we show how to prove Theorem~\ref{thm:PMonot0} and Theorem~\ref{theo:very-strong}.

\begin{proof}[Proof of Theorem~\ref{thm:PMonot0}] We are assuming that $W_1\in \mc{P}_{\leq}$, $W_2\in \mc{P}_{\geq}$, $W_2$ satisfies \eqref{eq:extraW2} and that $ \Sigma(W_2/W_1)=\infty$. Then the relation
\begin{equation}\label{eq:goal4}
\bb P(\exists\,\rho\in \mc{T}_f\,:\,\vec{X}_{\tau_{\rho}}^{\infty}\in \mc{A}^{\uparrow}_{\tau_{\rho}})=0\,,
\end{equation} 
can be deduced as in~\eqref{eq:upperNCZR} from the upper bound in~\eqref{eq:upperboundXYeps}.
\end{proof}

\begin{proof}[Proof of Theorem~\ref{theo:very-strong}]
We assume that~\eqref{eq:A2} holds. According to Lemma~\ref{lem:AntRW} the lower bound in~\eqref{eq:lowerbounddXdY} holds. We will proceed to show that
\begin{equation}\label{eq:goal1}
\bb P(\exists\,\rho\in \mc{T}_f\,:\,\vec{X}_{\tau_{\rho}}^{\infty}\in \mc{A}^{\uparrow}_{\tau_{\rho}})=1\,,
\end{equation} 
and hence proving that the vector of crossing numbers induced by $\vec{X}$ is eventually monotone.

Motivated by \eqref{eq:goal1} we define recursively:
\begin{enumerate}
	\item $\sigma_{0}=0$, $\bar{\tau}_0=0$
	\item If $\sigma_{k}=\infty$ then $\sigma_{k+1}=\infty$  and $\bar{\tau}_{k+1}=+\infty$,
	\subitem otherwise, $\bar{\tau}_{k+1}=\tau\circ\mc{S}_{\sigma_k}$
	\begin{equation}\label{eq:stoppAnt}
	\sigma_{k+1}=\inf\{n\geq \bar{\tau}_{k+1}\,:\,\vec{X}_{\bar{\tau}_{k+1}}^{n}\notin \mc{A}^{\uparrow}_{\bar{\tau}_{k+1}}\}\,.
	\end{equation}
\end{enumerate}


As in Section~\ref{sec:NCZR} it is enough to show that $\sum_{k\geq 0}\bb P_{\xi_{0}}(\sigma_{k}<\infty)<\infty$. Since the arguments employed to prove that mimic the ones already used in Section~\ref{sec:NCZR} we just sketch the main steps. 

Recall that, assuming that the graph $G$ has at least one circuit, the walker $\vec{Y}^\tau$ is well defined for all times.

Fix $m\in \bb N$. In~\eqref{eq:step1Ant} below we use~\eqref{eq:lowerbounddXdY} in the fourth line, whereas in the last line we use  that $\vec{Y}^{\tau}\in \mc{A}^{\uparrow}_{\tau}$ almost surely (cf. Lemma~\ref{lemma:suppYn}). We obtain
\begin{equation}\label{eq:step1Ant}
\begin{split}
\bb P_{\xi_{0}}(\sigma_{1}\geq m)&\geq \bb P_{\xi_{0}}(\vec{X}_{\tau}^m\in \mc{A}^{\uparrow}_{\tau})\\
&=\sum_{\pi_{\tau}^m \in \mc{A}^{\uparrow}_{\tau}}\bb P_{\xi_{0}}(\vec{X}_{\tau}^m=\pi_{\tau}^m)\\
&=\sum_{\pi_{\tau}^m \in \mc{A}^{\uparrow}_{\tau}}\frac{d\vec{X}}{d\vec{Y}^{\tau}}(\pi_{\tau}^m)\bb P_{\xi_{0}}(\vec{Y}_{\tau}^m=\pi_{\tau}^m)\\
&\geq \delta \bb P_{\xi_{0}}(\vec{Y}_{\tau}^m\in\mc{A}^{\uparrow}_{\tau})\\
&=\delta.
\end{split}
\end{equation}
 We can now conclude as in Section~\ref{sec:NCZR}.
 
If the graph $G$ has exactly one circuit, we leave to the reader to check that
 \begin{equation}
 \{\vec{X}_{\tau_{\rho}}^{\infty}\in \mc{A}^{\uparrow}_{\tau_{\rho}}\}=\{\vec{X}_{\tau_{\rho}}^{\infty}\in \mc{C}^{\uparrow}_{\tau_{\rho}}\}\,.
 \end{equation}

We assume now that $G$ has at least two circuits. 

From~\eqref{eq:A2} we have that $\Sigma(W_1^{-1})<+\infty.$ Recall that under our assumptions the lower bound in~\eqref{eq:lowerbounddXdY} holds. Our goal is to prove that
\begin{equation}\label{eq:goal2}
\bb P(\exists\,\rho\in \mc{T}_f\,:\,\vec{X}_{\tau_{\rho}}^{\infty}\in \mc{C}^{\uparrow}_{\tau_{\rho}})=1\,.
\end{equation} 

In order to prove \eqref{eq:goal2}  we will use Lemma~\ref{lemma:suppYn} which asserts that
\begin{equation*}
\bb P(\vec{Y}^{\tau_{\rho}}\in \mc{C}^{\uparrow}_{\tau_{\rho}})=1\,.
\end{equation*}
Recall that $\mc{C}^{\uparrow}_{\tau}\subset\mc{A}^{\uparrow}_{\tau}$. Then repeating the argument in \eqref{eq:step1Ant} with $\mc{C}^{\uparrow}_{\tau}$ instead {\color{blue}of} $\mc{A}^{\uparrow}_{\tau}$ we obtain the corresponding lower bound and as consequence the result.

\end{proof}

\subsection{The stopping time $\tau$ is finite}\label{sec:lem:taufinite}
In this section we prove Lemma~\ref{lem:taufinite}.

\begin{center}
	\textit{ Our first goal} is to prove that given $\rho\in \mc{T}_{f}$ the stopping time
	$
	\tau_{\rho}=\inf\{n\geq \rho\,:\,{\rm End}_n\subset S_n\}\,
	$
	is finite almost surely.
\end{center}

The proof is inspired by calculations that can already be found in~\cite[Section~3]{EFR2019}.

\textit{First case:} assume that $S_{\rho}=\emptyset$, i.e., for all pairs $(u,v)\in V^2$ we have that $|c_{\rho}(u,v)|\leq N.$ Using that $G$ has at least one circuit,  we fix a circuit $C^*$ in $G$ and consider the following event:
\begin{enumerate}
	\item For each $i\in[N]$, the walker $X_i$ follows a shortest path that connects $X_i(\rho)$ to $C^*.$ The crossing numbers decrease at most by $1$ for each walker.
	\item  Then the walker $X_i$ gives $N+3$ turns around $C^*.$ This increases the crossing numbers of $C^*$ by $N(N+3)$.
\end{enumerate}
Let $\rho'-\rho$ be the time needed for that event.  Recall that at time $\rho$ the crossing numbers for edges in the circuit are at least $-N$. When all the walkers arrive at $C^*$ these crossing numbers decrease by at most $-N$. At time $\rho'$ the crossing number of $C^*$  are increased by $N(N+3)$. Therefore, at time $\rho'$ the crossing numbers for the edges in the circuit $C^*$ are at least $N(N+3)-N-N=N^2+N$ which is greater than $N$ for any $N\geq 1$.  That is, $S_{\rho'}\neq \emptyset$ and ${\rm End}_{\rho'}\subset S_{\rho'}.$ Observe also that, since we started with bounded crossing numbers, the above event has a probability that is uniformly lower bounded by a positive constant  depending only on $N$ and the degrees of the graph (cf. \cite[Proposition~3.7]{EFR2019}).

\textit{Second case:} assume that $S_{\rho}\neq \emptyset.$ If ${\rm End}_{\rho}\subset S_{\rho}$ we can conclude because in that case  one simply has that $\tau_\rho=\rho$. Otherwise, the set $I=\{i\in [N]\,;\,X_{i}(\rho)\notin S_{\rho}\}$ is non empty. Without loss of generality we write $I=\{1,\cdots, K\}$ for $1\leq K\leq N$. We employ a strategy that works as follows:
\begin{enumerate}
\item First we force the walker $X_1$ to follow a shortest path that connects $X_1(\rho)$ to $S_{\rho}$ while all the other walkers are forced to stay put. This will ensure that there exists a time $n_1$ such that $X_1(\rho+n_1)\in S_{\rho+n_1} $.
\item In case $K\geq 2$, if $X_2 \in S_{\rho+n_1}$ there is nothing to do, otherwise we apply the first step to $X_2$ and $S_{\rho+n_1}$.
\item We iterate until the set $I$ is fully exhausted.
\end{enumerate} 
The above strategy ensures that there exists a time $n$ such that $X_i(\rho+n)\in S_{\rho+n}$ for all $i \in [N]$.

We now argue why the first step of the strategy works and can be employed with a probability that is lower bounded from below uniformly in the environment. A similar strategy has already been employed in~\cite[Lemma 3.2]{EFR2019} and therefore we refrain from giving a detailed proof here. Since $G$ is connected, there exists a shortest path $\pi=(\pi(1),\cdots,\pi(\ell))$ connecting $X_1(\rho)$ to $S_{\rho}$, and we force the walker $X_1$ to follow $\pi$. 
\begin{center}
Claim: $X_1(\rho+\ell)\in S_{\rho+\ell}.$
\end{center}  By the definition of a shortest path we have that $\pi(i)\notin S_{\rho}$ for $i=1,\cdots,\ell-1$ and $\pi(\ell)\in S_{\rho}$. This immediately implies, by the definition of $S_\rho$, that $|c_{\rho}(\pi(i),\pi(i+1))|\leq N$ for $i=1,\cdots,\ell-1$. Recall that there exists $w$ such that  $|c_\rho(\pi(\ell),w)|\geq N+1$. In particular, $w\neq \pi(i)$ for $i=1,\cdots,\ell-1$. As consequence, when the walker $X_1$ arrives at $\pi(\ell)$ we have that $|c_{\rho+\ell}(\pi(\ell),w)|\geq N+1$ because the edge $\{w,\pi(\ell)\}$ is not crossed by $X_1$ (recall that we force all other walkers to stay put). And the claim is proved since $X_1(\rho+\ell)=\pi(\ell).$  Since the crossing numbers of $\pi$ are bounded by $N$, the probability of following this path is uniformly lower bounded in the initial environment (but it depends on $N$).

 Now we apply the second step and when we finish the iteration the  final lower bound will still be uniform in the environment because the number of walkers is finite.

Similar calculations as in \eqref{eq:borelcantelli} show that $\tau_{\rho}$ is finite almost surely.

\begin{center}
	\textit{ Our second goal} is to prove that $v\in D_{\tau} \implies \deg_{D_{\tau}}^{\rm out}(v)\geq 1$. 
\end{center}

The proof relies on flow properties induced by the vector of crossing numbers. This idea was already used in~\cite{EFR2019}.

For each $j\in [N]$ we define the \textit{individual} flow by
\begin{equation*}
F_n^{(j)}(v)=\sum_{w:w\sim_Gv}c_n^{(j)}(v,w).
\end{equation*}
Then \cite[Lemma~3.4]{EFR2019} implies that $F_n^{(j)}(v)\in \{-1,0,+1\}$ for all times and we have the following cases:
\begin{enumerate}
	\item $F_n^{(j)}(v)=+1$ if $v$ is the starting point of $X_j$ but is not the end point.
	\item $F_n^{(j)}(v)=-1$ if $v$ is the end point of $X_j$ but is not at the same time the starting point.
	\item $F_n^{(j)}(v)=0$ in all other cases.
\end{enumerate}

We define the total flow at vertex $v\in V$ by
\begin{equation}\label{def:Totalflow}
F_n(v)=\sum_{j=1}^NF_n^{(j)}(v),
\end{equation}
and we directly have the following:
\begin{enumerate}
	\item $F_n(v) \in \llbracket-N,+N\rrbracket,$ for all $v\in V(G)$ and for all $n\geq 0.$
	\item $F_n(v)\geq 0$ if $v$ is not an end point of any of the walkers $X_j$ at time $n.$
\end{enumerate}

With the above facts at hand we proceed to show that $\deg_{D_{\tau}}^{\rm out}(v)\geq 1$ for all $v\in D_{\tau}$.  First of all note that by the definition of the stopping time $\tau,$ $S_{\tau}\subset V_{\tau}.$ 

\textit{First case:} $v \notin {\rm End}_{\tau}$. Then by definition of $V_{\tau}$ (cf. \eqref{def:D}) there exists $u\sim_Gv$ such that $c_{\tau}(v,u)\neq 0.$ The condition $F_{\tau}(v)\geq 0$ enforces that there exists $w\sim_G v$ with $c_{\tau}(v,w)>0.$ We prove that $w\in V_{\tau}$ resulting in $\deg_{D_{\tau}}^{\rm out}(v)\geq 1$:
\begin{itemize}
	\item If $w\in {\rm End}_{\tau}$ then $w\in V_{\tau}$ by definition.
	\item If $w\notin {\rm End}_{\tau}$ we use that $c_{\tau}(w,v)\neq 0$ which implies $w\in V_{\tau}.$
\end{itemize}  

\textit{Second case:} $v \in {\rm End}_{\tau}$.  Then by definition of $\tau$, $v\in S_{\tau}$. As a consequence, there exists $w\sim_Gv$ such that $|c_{\tau}(v,w)|\geq N+1.$ Since $F_n(v)\in \llbracket-N,N\rrbracket$, there exists $u\sim_Gv$ with $c_{\tau}(v,u)>0.$ The proof that $u\in V_{\tau}$ is the same as in the first case. That is, in any case we have that $\deg_{D_{\tau}}^{\rm out}(v)\geq 1.$

\begin{center}
	\textit{ Our third goal} is to prove: \\ each connected component of $D_{\tau_{\rho}}$ has at least one circuit.
\end{center}

This is a variation of a standard result in basic combinatorics which asserts that a graph $G$ with $\min_v \deg_G(v)\geq 2$ has at least one circuit. We prove the above claim for completeness. 

Fix a connected component $D\subseteq D_{\tau}$ and an arbitrary vertex $v_0 \in D$. Since $\deg_{D_{\tau}}^{\rm out}(v_0)\geq 1$, there exists $v_1\in D$ such that $v_1\sim_{D_{\tau}} v_0$. Note that $c_{\tau}(v_0,v_1)> 0$. In the same way we see that there exists $v_2\in D$ with $c_{\tau}(v_1,v_2)>0$. Since $c_{\tau}(v_2,v_1)= -c_{\tau}(v_1,v_2) <0 < c_{\tau}(v_0,v_1)$ we can conclude that $v_0\neq v_2$. One may now in the same fashion construct a sequence $v_0,v_1,v_2,\ldots$ of vertices in $D$ such that for all $i\geq 0$ one has that $c_{\tau}(v_i,v_{i+1}) >0$ and such that $v_i\neq v_{i+2}$, i.e., that sequence does not immediately backtrack. Since the oriented graph $D_{\tau}$ is finite, there will be eventually a repetition in the sequence, i.e., there exist two indices $i_1$ and $i_2$ with $|i_1-i_2|\geq 3$ such that $v_{i_1}=v_{i_2}$ so that the sequence $v_{i_1}, v_{i_1 +1},\ldots, v_{i_2}$ forms a circuit. Thus, we can conclude. 

\subsection{Estimative of the Radom-Nikodym derivative - Ant RW}\label{sec:lem:AntRW}
In this section we prove Lemma~\ref{lem:AntRW}.

The proof follows along similar lines as the one in Section~\ref{sec:lem:NCZR}. 
Again we will turn the expression $\frac{d\vec{X}}{d\vec{Y}}$ into an exponential of a sum, and estimate that exponential. In the case that $W_1\in \mc{P}_{\leq}$, $W_2\in \mc{P}_{\geq}$ (cf. \eqref{eq:defF}) the sum obtained in this way can be controlled by using the total flow of the crossing numbers. It then turns out that the new sum will behave like $\Sigma(W_2/W_1).$ For more general reinforcement functions we assume \eqref{eq:A2}.

Let $\pi_{\tau}^m\in \mc{A}_{\tau}^{\uparrow}$. In particular $\pi_{\tau}^m$ is a path, hence for every $k\in\llbracket \tau,m-1\rrbracket$, there exists a $j\in[N]$ such that $\pi_k(i)=\pi_{k+1}(i)$ for all $i\neq j$, and $(\pi_k(j),\pi_{k+1}(j))\in \vec{E}_\tau$. In the following computation we write $v_k^\pi$ instead of $\pi_j(k)$, i.e., $v_k^\pi$ is the only coordinate of $\pi_k$ which changes  from time $k$ to $k+1$. A straightforward computation shows that 
\begin{equation}\label{eq:exprdXdY}
\begin{split}
\frac{d\vec{X}}{d\vec{Y}}(\pi_{\tau}^{m})&=\prod_{k=\tau}^{m-1}\frac{\displaystyle\sum_{w;c_{k}(v^{\pi}_{k},w)>0}W_1(c_{k}(v^{\pi}_{k},w))}{\displaystyle\sum_{w;c_{k}(v^{\pi}_{k},w)>0}W_1(c_{k}(v^{\pi}_{k},w))+\displaystyle\sum_{w;c_{k}(v^{\pi}_{k},w)\leq 0}W_2(-c_{k}(v^{\pi}_{k},w))}\\
&=\prod_{k=\tau}^{m-1}\frac{1}{1+\frac{\displaystyle\sum_{w;c_{k}(v^{\pi}_{k},w)\leq 0}W_2(-c_{k}(v^{\pi}_{k},w))}{\displaystyle\sum_{w;c_{k}(v^{\pi}_{k},w)> 0}W_1(c_{k}(v^{\pi}_{k},w))}}.
\end{split}
\end{equation}
Now we split the analysis in cases, depending on the assumptions about $W_1$ and $W_2.$ We start with the simplest case.

\begin{center}
	\textit{Case 1:} $(W_1,W_2)$ satisfies \eqref{eq:A2}.
\end{center}

From~\eqref{eq:Upsilon}, and the inequality $1/(1+x)\geq e^{-x}$ applied to the right hand side in~\eqref{eq:exprdXdY} we immediately obtain that
\begin{equation*}
\begin{split}
\frac{d\vec{X}}{d\vec{Y}}(\pi_{\tau}^{m})&=\prod_{k=\tau}^{m-1}\frac{1}{1+\frac{\Upsilon^{-}(c_k(v_{k}^{\pi},\,\cdot\,))}{\Upsilon^{+}(c_k(v_{k}^{\pi},\,\cdot\,))}}\\
&\geq \exp\bigg(-\sum_{k=\tau}^{m-1}\frac{\Upsilon^{-}(c_k(v_{k}^{\pi},\,\cdot\,))}{\Upsilon^{+}(c_k(v_{k}^{\pi},\,\cdot\,))}\bigg)\,.\end{split}
\end{equation*}

It remains to estimate the sum inside the exponential.
To that end we write
\begin{equation}\label{eq:doublesum}
\sum_{k=\tau}^{m-1}\frac{\Upsilon^{-}(c_k(v_{k}^{\pi},\,\cdot\,))}{\Upsilon^{+}(c_k(v_{k}^{\pi},\,\cdot\,))}= \sum_{v\in V(D_\tau)} \sum_{k=\tau}^{m-1}\frac{\Upsilon^{-}(c_k(v,\,\cdot\,))}{\Upsilon^{+}(c_k(v,\,\cdot\,))}\mathds{1}_{\{v^{\pi}_{k}= v\}}\,.
\end{equation}

Recall the definition of $\Phi$ in Definition~\eqref{def:Phi}.
Fix a vertex $v$ and let $n_1(v)$, $n_2(v)$, $\ldots$ be the sequence of times at which the path $\pi_{\tau}^{m}$ visits $v$. Since the state space of $\pi_{\tau}^m$ is $D_{\tau}^N$ we can conclude that $(c_{n_i(v)}(v,\cdot))_i$ belongs to $\Phi$. Hence, by Assumption~\eqref{eq:A2} the inner sum in~\eqref{eq:doublesum} is uniformly bounded from above, and since $|V|$ is finite the same is true for the sum on the left hand side of~\eqref{eq:doublesum}. Hence, we can conclude.

\begin{center}
	\textit{Case 2:} $W_1\in \mc{P}_{\leq}$, $W_2\in \mc{P}_{\geq}$, $W_2$ satisfies \eqref{eq:extraW2} and $ \Sigma(W_2/W_1)=\infty$.
\end{center}

Define the positive and negative flow (cf. \eqref{def:Totalflow}) at $v^{\pi}_{k}$ by
\begin{equation*}
\begin{split}
F_{k}^{+}(v^{\pi}_{k})&=\displaystyle\sum_{w;c_{k}(v^{\pi}_{k},w)> 0}c_{k}(v^{\pi}_{k},w)\,,\\
F_{k}^{-}(v^{\pi}_{k})&=\displaystyle\sum_{w;c_{k}(v^{\pi}_{k},w)\leq 0}-c_{k}(v^{\pi}_{k},w)\,.
\end{split}
\end{equation*}
By \eqref{eq:defF} there exists a constant $C=C(\Delta(G),W_1,W_2)$ such that
\begin{equation}\label{eq:boundW}
\begin{split}
 \displaystyle\sum_{w;c_{k}(v^{\pi}_{k},w)> 0}W_1(c_{k}(v^{\pi}_{k},w))&\leq CW_1(F_{k}^{+}(v^{\pi}_{k}))\,,\\
 \displaystyle\sum_{w;c_{k}(v^{\pi}_{k},w)\leq 0}W_2(-c_{k}(v^{\pi}_{k},w))&\geq C^{-1}W_2(F_{k}^{-}(v^{\pi}_{k}))\,.
\end{split}
\end{equation}

 Using \eqref{eq:boundW} in \eqref{eq:exprdXdY}, the relation $F_k^+(v_{k}^{\pi})-F_k^-(v_{k}^{\pi})=F_k(v_{k}^{\pi})$ and that $W_2$ satisfies \eqref{eq:extraW2}, we see that there exists a constant $\bar{\mc C}$ such that
\begin{equation*}
\frac{d\vec{X}}{d\vec{Y}}(\pi_{\tau}^{m})\leq \prod_{k=\tau}^{m-1}\frac{1}{1+C^{-2}\frac{W_2(F_k^{-}(v_{k}^{\pi}))}{W_1(F_k^{+}(v_{k}^{\pi}))}}\leq \prod_{k=\tau}^{m-1}\frac{1}{1+\bar{\mc C}\frac{W_2(F_k^{+}(v_{k}^{\pi}))}{W_1(F_k^{+}(v_{k}^{\pi}))}}.
\end{equation*}
Observe that for $x\in (0,2]$ it holds that $1/(1+x)\leq e^{-x/2}.$ If $x>2$ we have that $1/(1+x)\leq 1/3$. Define the set of indices
\begin{equation*}
I_{m}=\{\tau\leq k\leq m-1\,:\,\bar{\mc C}\frac{W_2(F_k^{+}(v_{k}^{\pi}))}{W_1(F_k^{+}(v_{k}^{\pi}))}\leq 2\}\,.
\end{equation*}
Therefore,
\begin{equation}\label{eq:finalupper}
\frac{d\vec{X}}{d\vec{Y}}(\pi_{\tau}^{m})\leq 3^{-|I_m^{\complement}|}\exp\bigg(-\frac{1}{2}\sum_{k\in I_{m}} \bar{\mc C}\frac{W_2(F_k^{+}(v_{k}^{\pi}))}{W_1(F_k^{+}(v_{k}^{\pi}))}\bigg).
\end{equation}
As $m\to \infty$ then either $|I_m|\to \infty$ or $|I_m^{\complement}|\to \infty$ (or both tend to infinity). Since, $\Sigma(W_2/W_1)=\infty$, using similar arguments as in~\eqref{eq:doublesum} we can conclude in any case.

\subsection{The support of $\vec{Y}^{\tau}$}\label{sec:lemma:suppYn}
In this section we prove Lemma~\ref{lemma:suppYn}.

The strategy of the proof is to use a collection of independent balls-in-bins processes (cf. Section~\ref{sec:balls-in-bins}) to provide an alternative construction of the process $\vec{Y}^{\tau}.$ Then we use Theorem~\ref{thm:Rubin} to describe the long time behaviour of $\vec{Y}^{\tau}$.

For each $v \in V_{\tau}$ we consider the balls-in-bins process $\vec{\eta}^{v}:=(\eta^{v}_{w}(n);(v,w)\in \vec{E}_{\tau})_{n\in \bb N}$ with initial number of balls given by $\eta^{v}_{w}(0)=c_{\tau}(v,w)$ and reinforcement function $W_1.$ Observe that $\eta^{v}(w)(0)\geq 1$ by definition of $D_{\tau}$ and that  the process $\vec{\eta}^{v}$ consists of $\deg_{D_{\tau}}^{\rm out}(v)\geq 1$ bins.

For the alternative construction of the process $\vec{Y}^{\tau}$ from time $\tau$ onwards assume that $\vec{Y}^{\tau}(\tau)=\vec{v} \in D_{\tau}^N$. We first choose $i\in [N]$ uniformly at random. It means that the walker $Y_i$ will jump. Since $Y_i(\tau)=v_i$ we throw a ball in a bin according to the law of $\vec{\eta}^{v_i}$. If this ball is thrown in the bin corresponding to $w$ then we set $Y^{\tau}_{i}(\tau+1)=w.$ In particular, this means that $\eta^{v}_{w}(1)=\eta^{v}_{w}(0)+1=c_{\tau,\tau}(v,w)+1=c_{\tau,\tau+1}(v,w).$  We continue the process in this way. It is easy to see from the definition of ball-in-bins dynamics (cf. Section~\ref{sec:balls-in-bins}) that we generate a process with the same law as the process $\vec{Y}^{\tau}$ defined in Section~\ref{sec:proofAntRW}. 

We proceed to prove Lemma~\ref{lemma:suppYn}. From now on we drop the dependency on $\tau.$ Recall the definition of $\mc{D^{\infty}}$ in~\eqref{def:Dinfty}.

Assume that $G$ has exactly one circuit. We claim that each walker $Y_j$, $1\leq j\leq N$, is trapped in the unique circuit $C_*$ of $G$. Indeed, since $Y_j$ is non-backtracking and the graph $G$ is finite, it must close a circuit which is $C_*$. If $Y_j$ would leave $C_*$ then since it is non-backtracking it would need to eventually close a circuit which however would be different from $C_*$. This is impossible by the assumption that $G$ has only one circuit. Note at that point that all balls-in-bins process  that we use to emulate the dynamics of $\vec{Y}$ will have only one bin, since each bin in the construction of the processes $\vec{Y}$ corresponds to an edge whose corresponding crossing number is positive  and therefore the trajectory of $\vec{Y}$ is restricted to a circuit. In particular, the monopolistic and leadership events coincide with the realization of the balls-in-bins processes.

Assume now that $G$ has at least two circuits. Our goal is to describe the graph $\mc{D}^{\infty}$. Our first step is to prove that $\mc{D}^{\infty}$ is a union of circuits. We further specialize this description depending on the regime of $W_1$.

\textit{First step:} $\mc{D}^{\infty}$ is the union of circuits.
 Let $\Gamma$ be the collection of circuits of $\mc{D}^\infty$ that are crossed infinitely often. We also define 
\begin{equation*}
\Gamma_{v}=\{C\in \Gamma\,:\,v\in C\}\,.
\end{equation*} 
\begin{center}
	\textit{Claim:} we have that $\Gamma_{v}\neq \emptyset$ for every $v\in \mc{D}^{\infty}$.
\end{center}
Indeed, if $v\in \mc{D}^{\infty}$ there exists a walker $Y_j$ that visits $v$ infinitely often since the number of walkers are finite. If $Y_j(n)=v$ the first return to $v$ after time $n$ generates a circuit because $Y_j$ is non-backtracking: it uses an outer edge to leave $v$, and uses an inner edge to return to $v$. That is, each visit of $Y_j$ to $v$ corresponds to a complete visit to some circuit containing $v$. Since the number of circuits is finite there exists a circuit $C_v$ with $v\in C_v$ that is visited infinitely often by $Y_j,$ which is enough to conclude the argument. As consequence of the claim, we have that $\mc{D}^{\infty}=\cup_{C\in \Gamma}C.$ 

\textit{Second step:} here we need to assume that $\Sigma(W_{1}^{-1})<\infty.$

Our goal is to prove that the union $\mc{D}^{\infty}=\cup_{C\in \Gamma}C$ {\color{blue} is } disjoint. To do that, it is enough to
prove that $|\Gamma_v|=1$ for any $v\in \mc{D}^{\infty}$. Indeed, let $C,C'\in \Gamma$ with  $C\cap C'\neq \emptyset$. For $v\in C\cap C'$, $|\Gamma_v|=1$ immediately implies that $C=C'.$

Now we prove that $|\Gamma_v|=1$ for a fixed $v\in \mc{D}^{\infty}.$ Let $C_1,\,C_2\in \Gamma$ with $v\in C_1\cap C_2$ and $C_1\neq C_2.$  Then there exist vertices $w_1$ and $w_2$ such that $(v,w_1)$ is an edge of $C_1$ but not of $C_2$ and $(v,w_2)$ is an edge of $C_2$ but not of $C_1$. In particular, $(v,w_1)\in \vec{\mc{E}}^{\infty}$ and $(v,w_2)\in \vec{\mc{E}}^{\infty}$ with $w_1\neq w_2$ (recall that our graph does not have multiple edges).   However, each visit of $\vec{Y}$ to $v$ corresponds to a realization of a balls-in-bins process with reinforcement function $W_1$. Since $\Sigma(W_1^{-1})<\infty$, Theorem~\ref{thm:Rubin} asserts that we are in the monopolistic regime. This is a contradiction with the fact that the bin corresponding to $(v,w_1)$ receives infinitely many balls as well as the bin corresponding to $(v,w_2)$. Hence, $C_1=C_2$.

\appendix

\section{Proof of Proposition~\ref{prop:examples}}
\label{sec:proofexamples}

\begin{center}
 Proof part \rm{(a)}:
\end{center}

 By assumption $\sum_{k\geq 0}W_1(k)^{-1}<\infty$ and $W_2$ is non-increasing. Consider
an element $(\phi_n)_{n\in\bb N}\in\Phi$, and denote by $d$ its dimension, i.e., $\phi_n\in\bb Z^d$ for all $n$. Since $W_2$ is non-increasing one has in particular that $W_2(-k)\leq W(0)$ for all $k\leq 0$. Hence,
\begin{equation}\label{eq:W2nonincreasing}
\Upsilon^{-}(\phi_n)=\sum_{i:\phi(i)\leq 0}W_2(-\phi_n(i))\leq d W_2(0),
\end{equation}  for all $n\in\bb N$. To estimate the positive part, we first note that for any $n$ and $i$ such that $\phi_n(i) >0$ one has that 
\begin{equation*}
\Upsilon^{+}(\phi_n)=\sum_{j:\phi(j)>0} W_1(\phi_n(j)) \geq W_1(\phi_n(i)).
\end{equation*}
Moreover, by the definition of the sequence $(\phi_n)_{n\in\bb N}$ for any $n\geq 0$ there exists $i=i(n)\in\{1,2,\ldots,d\}$  such that $\phi_{n+1}(i) =\phi_{n}(i)+1$. Thus, we see that
\begin{equation}\label{eq:W1summable}
\sum_{n=1}^{\infty} \frac{1}{\Upsilon^{+}(\phi_n)}\leq 
\sum_{n=1}^{\infty} \frac{1}{W_1(\phi_n(i(n)))}.
\end{equation}
To continue, for $k\in\bb N$ we define the quantity $L_k$ via
\begin{equation*}
L_k= \big|\{n\in\bb N:\, \phi_n(i(n))= k\}\big|,
\end{equation*}
and we note that $L_k\leq d$, since the vector $\phi_n$ has $d$ components and one entry is updated at each time enforcing that at most in $d$ instances of time one has a component at level $k$. We therefore can conclude that
\begin{equation*}
\begin{aligned}
\sum_{n=1}^{\infty}\frac{1}{W_1(\phi_n(i(n)))}&=\sum_{n=1}^{\infty}\sum_{k\geq 0}\frac{1}{W_1(\phi_n(i(n)))}\one_{\{\phi_n(i(n))=k\}}\\
&=\sum_{k\geq 0}\frac{L_k}{W_1(k)}\\
& \leq d\sum_{k=1}^{\infty}\frac{1}{W_1(k)} <\infty,
\end{aligned}
\end{equation*}
where the last inequality uses that $\Sigma(W_1^{-1})<\infty$. Combining this with~\eqref{eq:W2nonincreasing} and~\eqref{eq:W1summable} allows us to conclude that \eqref{eq:A2} holds.

\begin{center}
Proof of part \rm{(b)}.
\end{center}

  We show that~\eqref{eq:A2} is satisfied. To that end we denote by $C$ a proportionality constant that might change from occurrence to occurrence.  We consider a function $(\phi_n)_{n\in\bb N}\in\Phi$ as in the first part of the proof. We note that for any index $\alpha >0$, any $j\in\bb N$
for any collection of positive numbers $a_1,\ldots, a_j$ one has that 
\begin{equation}\label{eq:equiv}
\frac{1}{C}\Big(\sum_{i=1}^{j} a_i\Big)^\alpha\leq \sum_{i=1}^{j}a_i^\alpha \leq C \Big(\sum_{i=1}^{j} a_i\Big)^\alpha,
\end{equation}
where $C$ above only depends on $\alpha$ and $j$.
To continue, we define
\begin{equation*}
\mc F_+(\phi_n):=\sum_{\substack{i:\, \phi_n(i)>0}}\phi_n(i),\quad
\mc F_-(\phi_n):=\sum_{\substack{i:\, \phi_n(i)\leq 0}}\phi_n(i),
\end{equation*}
and we note that the second item in Definition~\ref{def:Phi} implies that for any $n$ one has that $\mc F_+(\phi_n)+\mc F_-(\phi_n)\in \llbracket-N,+N\rrbracket$. Therefore,~\eqref{eq:equiv} yields the estimate
\begin{equation}
\frac{\Upsilon^{-}(\phi_n)}{\Upsilon^{+}(\phi_n)}\leq C \frac{(-\mc F_-(\phi_n))^q}{\mc F_+(\phi_n)^p}\leq C \frac{(N+\mc F_+(\phi_n))^q}{\mc F_+(\phi_n)^p} \leq C \frac{1}{\mc F_+(\phi_n)^{p-q}}\,,
\end{equation} 
To finish the proof it only remains to note that the sequence $(\mc F_+(\phi_n))_{n\in\bb N}$ is a sequence of positive numbers satisfying $\mc F_+(\phi_{n+1}) = \mc F_+(\phi_n)+1$ for all $n\in\bb N$.

\begin{center}
Proof of part \rm{(c)}.
\end{center}

First of all note that by the monotonicity of the logarithm, we have that for some constant $C$
\begin{equation*}
\Upsilon^{-}(\phi)\leq C  \log\Big(-\mc F_-(\phi_n)\Big)^q.
\end{equation*}
Moreover, one may show that there exists a universal constant $C$ such that
\begin{equation}
\sum_{i:\, \phi_n(i) >0} W_1(\phi_n(i)) \geq C W_1(\mc F_+(\phi_n)),
\end{equation}
and we may conclude as in part~\rm{(b)}.

\begin{center}
Proof of part \rm{(d)}.
\end{center}

The assumptions $W_1\in \mc{P}_{\geq}$, $W_2\in \mc{P}_{\leq}$, and $W_2$ satisfies \eqref{eq:extraW2} guarantees that the proof of part~\rm{(d)} is exactly the same as in part~\rm{(b)}.

\begin{center}
Proof of part \rm{(e)}.
\end{center}

Consider again an element $(\phi_n)_{n\in\bb N}\in\Phi$, and denote by $d$ its dimension, i.e., $\phi_n\in\bb Z^d$ for all $n$. Recall that $W_2(k)\leq e^{\alpha k}$ for all $k\in \bb N$, and note that $k\mapsto e^{\alpha k}$ is increasing. Then we have, for all $n\in\bb N$, that
\begin{equation*}
\begin{split}
\Upsilon^{-}(\phi_n)&=\sum_{i:\phi(i)\leq 0}W_2(-\phi_n(i))\\
&\leq \sum_{i:\phi(i)\leq 0}\exp(\alpha(-\phi_n(i)))\\
&\leq d\exp\bigg(\alpha\sum_{i:\phi(i)\leq 0}(-\phi_n(i))\bigg)\,.
\end{split}
\end{equation*} 

Using that $W_1(k)=e^{\beta k}$ is convex, we obtain that
\begin{equation}
\begin{split}
\Upsilon^{+}(\phi_n)&=\sum_{i:\phi(i)> 0}W_1(\phi_n(i))\\
&= \sum_{i:\phi(i)> 0}\exp(\beta\phi_n(i))\\
&\geq d\exp\bigg(\frac{\beta}{d}\sum_{i:\phi(i)> 0}\phi_n(i)\bigg)\,.
\end{split}
\end{equation}  

Therefore, as in part \rm{(b)} we obtain
\begin{equation}\label{eq:finaleq}
\frac{\Upsilon^{-}(\phi_n)}{\Upsilon^{+}(\phi_n)}\leq  \frac{\exp(\alpha(N+\mc{F}_+(\phi_n)))}{\exp\big(\frac{\beta}{d}(\mc{F}_+(\phi_n))\big)}
\end{equation}

Since $N$ is fixed and $\mc{F}_+(\phi_n)\to \infty$, as $n\to \infty$, the sum of~\eqref{eq:finaleq} over $n$ is finite if $\beta>d\alpha$. We are done because $\Delta(G)\geq d$.

\section*{Acknowledgements}
 D.~E. gratefully acknowledges financial support 
from the National Council for Scientific and Technological Development - CNPq via a 
Universal grant 409259/2018-7, and a Bolsa de Produtividade 303520/2019-1. D.~E. moreover acknowledges support by the Serrapilheira Institute which supported this work (grant number Serra - R-2011-37582). G.~R. was supported by a Capes/PNPD fellowship 888887.313738/2019-00. The authors are grateful to T. Franco and A. Teixeira for fruitful discussions about the topic of the paper. The authors are also grateful to the anonymous referees for their valuable input.

\bibliographystyle{plain}
\bibliography{bibliografia}

\begin{thebibliography}{10}

\bibitem{Angel_2007}
A~G Angel, M~R Evans, E~Levine, and D~Mukamel.
\newblock Criticality and condensation in a non-conserving zero-range process.
\newblock {\em Journal of Statistical Mechanics: Theory and Experiment},
  2007(08):P08017--P08017, 2007.

\bibitem{Arnold1981}
Ludwig Arnold.
\newblock Mathematical models of chemical reactions.
\newblock In Michiel Hazewinkel and Jan~C. Willems, editors, {\em Stochastic
  Systems: The Mathematics of Filtering and Identification and Applications},
  pages 111--134, Dordrecht, 1981. Springer Netherlands.

\bibitem{Blount91}
Douglas Blount.
\newblock Comparison of stochastic and deterministic models of a linear
  chemical reaction with diffusion.
\newblock {\em Ann. Probab.}, 19(4):1440--1462, 1991.

\bibitem{Blount92}
Douglas Blount.
\newblock Law of large numbers in the supremum norm for a chemical reaction
  with diffusion.
\newblock {\em Ann. Appl. Probab.}, 2(1):131--141, 1992.

\bibitem{cotar2017}
Codina Cotar and Debleena Thacker.
\newblock Edge- and vertex-reinforced random walks with super-linear
  reinforcement on infinite graphs.
\newblock {\em Ann. Probab.}, 45(4):2655--2706, 07 2017.

\bibitem{EFR2019}
Dirk {Erhard}, Tertuliano {Franco}, and Guilherme {Reis}.
\newblock {The Directed Edge Reinforced Random Walk: The Ant Mill Phenomenon}.
\newblock \url{https://arxiv.org/abs/1911.07295}, 2019.

\bibitem{Franco2012}
Tertuliano Franco and Pablo Groisman.
\newblock A particle system with explosions: Law of large numbers for the
  density of particles and the blow-up time.
\newblock {\em Journal of Statistical Physics}, 149(4):629--642, 2012.

\bibitem{Kious2020}
Daniel {Kious}, C{\'e}cile {Mailler}, and Bruno {Schapira}.
\newblock {Finding geodesics on graphs using reinforcement learning}.
\newblock {\em arXiv e-prints}, October 2020.

\bibitem{LimicTarres2007}
V.~Limic and P.~Tarr\`es.
\newblock Attracting edge and strongly edge reinforced walks.
\newblock {\em Ann. Probab.}, 35(5):1783--1806, 2007.

\bibitem{Limic2003}
Vlada Limic.
\newblock {Attracting edge property for a class of reinforced random walks}.
\newblock {\em The Annals of Probability}, 31(3):1615 -- 1654, 2003.

\bibitem{ma_xia_yang_2017}
Tianren Ma, Zhengyou Xia, and Fan Yang.
\newblock An ant colony random walk algorithm for overlapping community
  detection.
\newblock {\em Lecture Notes in Computer Science Intelligent Data Engineering
  and Automated Learning--IDEAL 2017}, pages 20 –-- 26, 2017.

\bibitem{merkl}
Franz Merkl and Silke W.~W. Rolles.
\newblock {\em Linearly edge-reinforced random walks}, volume Volume 48 of {\em
  Lecture Notes--Monograph Series}, pages 66--77.
\newblock Institute of Mathematical Statistics, Beachwood, Ohio, USA, 2006.

\bibitem{2005Oliveira}
Roberto Oliveira.
\newblock Balls-in-bins processes with feedback and {B}rownian motion.
\newblock {\em Combin. Probab. Comput.}, 17(1):87--110, 2008.

\bibitem{Pemantle1988}
R.~Pemantle.
\newblock Phase transition in reinforced random walk and {RWRE} on trees.
\newblock {\em Ann. Probab.}, 16(3):1229--1241, 1988.

\bibitem{pemantle2007}
Robin Pemantle.
\newblock A survey of random processes with reinforcement.
\newblock {\em Probab. Surveys}, 4:1--79, 2007.

\bibitem{Quittner2007}
Pavol Quittner and Philippe Souplet.
\newblock {\em Superlinear parabolic problems}.
\newblock Birkh\"{a}user Advanced Texts: Basler Lehrb\"{u}cher. [Birkh\"{a}user
  Advanced Texts: Basel Textbooks]. Birkh\"{a}user Verlag, Basel, 2007.
\newblock Blow-up, global existence and steady states.

\end{thebibliography}

\end{document}